\documentclass[11pt,a4paper]{article}

\usepackage{pdfsync}
\usepackage[numbers]{natbib}

\usepackage{latexsym,amsfonts,amsmath,amssymb,bbm}	
\usepackage{url} 
\usepackage{amsthm}
\usepackage{color}
\usepackage{mathrsfs}  
\usepackage[normalem]{ulem}

\newtheorem{theorem}{Theorem}[section]

\newtheorem{lemma}[theorem]{Lemma} 

\numberwithin{equation}{section}

\DeclareMathOperator*{\ii}{\mathcal{{ \rm i }}}

\newcommand{\sbinom}[2]{{\textstyle{\binom{#1}{#2}}}}
\newcommand{\satop}[2]{\stackrel{\scriptstyle{#1}}{\scriptstyle{#2}}}

\newcommand{\bsUpsilon}{{\boldsymbol{\Upsilon}}}
\newcommand{\bsalpha}{{\boldsymbol{\alpha}}}

\newcommand{\bstau}{{\boldsymbol{\tau}}}
\newcommand{\bsgamma}{{\boldsymbol{\gamma}}}
\newcommand{\bsDelta}{{\boldsymbol{\Delta}}}

\newcommand{\bsnu}{{\boldsymbol{\nu}}}
\newcommand{\bszero}{{\boldsymbol{0}}}
\newcommand{\bshalf}{{\boldsymbol{\tfrac{1}{2}}}}

\newcommand{\bsb}{{\boldsymbol{b}}}

\newcommand{\bse}{{\boldsymbol{e}}}

\newcommand{\bsm}{{\boldsymbol{m}}}

\newcommand{\bst}{{\boldsymbol{t}}}

\newcommand{\bsx}{{\boldsymbol{x}}}
\newcommand{\bsxt}{\widetilde{{\boldsymbol{x}}}}
\newcommand{\bsy}{{\boldsymbol{y}}}
\newcommand{\bsz}{{\boldsymbol{z}}}

\newcommand{\bseta}{{\boldsymbol{\eta}}}

\newcommand{\bfn}{{\vec{\mathbf{n}}}}

\newcommand{\rd}{{\mathrm{d}}}
\newcommand{\bbA}{{\mathbb{A}}}

\newcommand{\bbS}{{\mathbb{S}}}
\newcommand{\bbR}{{\mathbb{R}}}
\newcommand{\bbZ}{{\mathbb{Z}}}
\newcommand{\bbN}{{\mathbb{N}}}
\newcommand{\bbE}{{\mathbb{E}}}
\newcommand{\bbP}{{\mathbb{P}}}

\newcommand{\calI}{{\mathcal{I}}}
\newcommand{\calO}{{\mathcal{O}}}
\newcommand{\calT}{{\mathcal{T}}}

\newcommand{\setu}{{\mathrm{\mathfrak{u}}}}
\newcommand{\setv}{{\mathrm{\mathfrak{v}}}}
\newcommand{\indx}{{\mathfrak F}}
\newcommand{\supp}{{\mathrm{supp}}}
\newcommand{\mask}[1]{{}}

\newcommand{\bsys}{\bsy_{\{1:s\}}}
\newcommand{\La}{\mathcal{L}}
\newcommand{\Lad}{\mathcal{B}}
\newcommand{\Lads}{\mathcal{B}_s}
\newcommand{\Lady}{\mathcal{B}(\bsy)}
\newcommand{\Ladys}{\mathcal{B}_s(\bsys)}

\newcommand{\Ladinvy}{\left[\mathcal{B}(\bsy)\right]^{-1}}

\newcommand{\Laj}{\mathcal{T}_{j}}
\newcommand{\Lazs}{\mathcal{A}_{0}}
\newcommand{\Lajs}{\mathcal{A}_{j}}
\newcommand{\M}{\mathcal{M}}
\newcommand{\bilin}{\mathscr{B}}
\newcommand{\antilin}{\mathscr{G}}
\newcommand{\bilins}{\mathcal{S}}
\newcommand{\antilins}{\mathcal{\ell}}

\definecolor{darkred}{RGB}{139,0,0}
\definecolor{darkgreen}{RGB}{0,100,0}
\definecolor{darkmagenta}{RGB}{180,0,180}
\definecolor{darkblue}{RGB}{0,0,190}

\title{Quasi-Monte Carlo finite element analysis for wave propagation in heterogeneous random media}

\author{M. Ganesh\footnotemark[2],\; Frances Y.~Kuo\footnotemark[3]\;\; and\; Ian H.~Sloan\footnotemark[3]}

\renewcommand{\thefootnote}{\fnsymbol{footnote}}

\date{January 2021}

\begin{document}

\footnotetext[2]{Department of Applied Mathematics \& Statistics, Colorado School of Mines, Golden, Colorado
80401, USA ({\tt mganesh@mines.edu}).}

\footnotetext[3]{School of Mathematics and Statistics, University of New
South Wales, Sydney NSW 2052, Australia ({\tt f.kuo@unsw.edu.au}, {\tt
i.sloan@unsw.edu.au}).}

\renewcommand{\thefootnote}{\arabic{footnote}}

\maketitle

\begin{abstract}
We propose and analyze a quasi-Monte Carlo (QMC) algorithm for efficient
simulation of wave propagation modeled by the Helmholtz equation in a
bounded region in which the refractive index is random and spatially
heterogenous. Our focus is on the case in which the region can contain
multiple wavelengths. We bypass the usual sign-indefiniteness of the
Helmholtz problem by switching to an alternative sign-definite formulation
recently developed by Ganesh and Morgenstern (Numerical Algorithms,
{\bf{83}}, 1441--1487, 2020). The price to pay is that the regularity
analysis required for QMC methods becomes much more technical.
Nevertheless we obtain a complete analysis with error comprising
stochastic dimension truncation error, finite element error and cubature
error, with results comparable to those obtained for the diffusion
problem. \\


\end{abstract}


\section{Introduction}\label{sec:intro}

This paper is concerned with a new algorithm and associated numerical
analysis for efficient simulation of wave propagation modeled by the
Helmholtz equation in a bounded region in which the refractive index is
random and spatially heterogenous. The wave is induced by an impinging
incident wave, and our focus is on the case in which the region can
contain multiple wavelengths. The main aim of this article is to compute
the expected value of a linear functional of the resulting wave field by
the use of a well designed \emph{Quasi-Monte Carlo} (\emph{QMC}) method
\cite{DKS13,DP10,Nie92,SJ94}, and to bound the resulting error.

The design and analysis of QMC methods has been well studied for the
classical diffusion problem, see for example
\cite{GGKSS19,GKNSSS15,GKNSS18b,HPS17,KN16,KSS12}. However, it is well
known that the standard Galerkin variational formulation for the Helmholtz
partial differential equation (PDE) lacks positive definiteness unless the
wavelength is relatively large compared to the region. The resulting lack
of coercivity (or sign-definiteness) rules out the standard QMC analysis
that has recently been used successfully for strongly elliptic diffusion
problems with random input. The analysis of QMC methods has also been
extended to a general class of operator equations, see \cite{Sch13} and
subsequent papers, e.g., \cite{DKLNS14,DLS16,Gan18,GHS18}. These papers
include the case of the Helmholtz equation, under an appropriate inf-sup
condition on the standard Galerkin variational formulation and assumptions
on the wavelength and the random component of the refractive index.

In this paper we bypass the sign-indefiniteness problem in a different
way, by using the recently developed sign-definite deterministic
formulation of Ganesh and Morgenstern~\cite{GanMor20}. In the present work
that analysis is extended to include randomness in the heterogeneous
refractive index. In the resulting QMC analysis there is a price to pay
for using the sign-definite formulation of \cite{GanMor20}, in that the
analysis of regularity with respect to the stochastic variables becomes
complicated, and a new approach is needed. On the other hand it has the
advantage that the space discretization can be carried out with the
standard Galerkin scheme without any threat of instability.

Precisely, we study the wave propagation problem in a bounded domain $D
\subset \mathbb{R}^d$, for $d=2,3$ with Lipschitz boundary $\partial D$.
The incident wave is of wavelength $\lambda=2\pi/k$, where $k$ is the
positive wavenumber, and our interest extends to wavelengths $\lambda$
smaller than $L$, where $L$ is a characteristic length of $D$, or
equivalently to $kL>2\pi$. The square of the refractive index,
$n(\bsx,\omega)$, in the interior of $D$ may be spatially varying, and is
also random, as described below.

For a deterministic forcing function $f \in L^2(D)$ and boundary data
$g\in L^2(\partial D)$, and for almost all elementary events $\omega$ in
the probability space $(\Omega,\mathcal{A},\bbP)$, the unknown field
$u(\cdot, \omega) \in H^1(D)$ is assumed to satisfy the Helmholtz PDE and
an absorbing boundary condition
\begin{equation}\label{eq:pde_omega}
 (\La u)(\bsx, \omega) = -f(\bsx)\,, \quad \bsx \in  D\,,
 \qquad  \text{and} \qquad
 \frac{\partial u}{\partial \bfn}(\widetilde{\bsx}, \omega)-\ii\, k\,u(\widetilde{\bsx} , \omega)
 = g(\widetilde{\bsx} )\,, \quad   \widetilde{\bsx} \in \partial D\,,
\end{equation}
where the \emph{stochastic Helmholtz operator} is given by
\begin{equation}\label{eq:pdeoper_omega}
 (\La u)(\bsx, \omega) \,:=\, \Delta u(\bsx, \omega) + k^2\,n(\bsx, \omega)\,u(\bsx, \omega)\,,
 \qquad \bsx \in  D\,, \quad \omega\in (\Omega,\mathcal{A},\bbP)\,.
\end{equation}
Here $\bfn = \bfn(\widetilde{\bsx})$ is the outward-pointing unit normal
vector, defined almost everywhere on the surface $\partial D$ of the
Lipschitz domain $D$. The system \eqref{eq:pde_omega} is a well known
model for a wide class of applications, including acoustic,
electromagnetic, and seismic wave propagation in heterogeneous
media~\cite{colton:inverse, Ihlenburg:1998,nedlec:book}. The boundary
condition in~\eqref{eq:pde_omega} is standard for the interior wave
propagation model and, as described in~\cite{GanMor20} and references
therein, it can be either considered as an approximation of the Sommerfeld
radiation condition  occurring in the unbounded medium counterpart of our
model, or can be used as an interface condition in the
heterogeneous-homogeneous coupled wave propagation
model~\cite{dgs2020,GanMor16}.

The random coefficient $n(\bsx,\omega)$, $\bsx \in D$, is taken in this
article to be parameterized by an infinite-dimensional vector
$\bsy(\omega) = (y_1(\omega),y_2(\omega),\ldots)$. For a fixed realization
$\omega \in \Omega$, we denote the corresponding deterministic parametric
coefficient by $n(\bsx,\bsy)$ and the associated solution to the above PDE
model by $u(\bsx,\bsy)$. We assume that the parameter $\bsy$ is uniformly
distributed on
\begin{align*}
  U \,:=\, [-\tfrac{1}{2},\tfrac{1}{2}]^\bbN\,,
\end{align*}
with the uniform probability measure $\mu(\rd\bsy) = \bigotimes_{j\geq 1}
\rd y_j = \rd\bsy$, where   $\bbN$ is the set of positive integers.

The non-negative, uncertain, coefficient $n(\bsx,\bsy)$ is assumed to be
expressible as a mean field $n_0(\bsx)$  plus a perturbation,
\begin{align} \label{eq:axy-unif}
  n(\bsx,\bsy)
  \,=\, n_0(\bsx) + \sum_{j\geq 1} y_j\, \psi_j(\bsx)\,,
  \qquad
  \bsx\in D\,, \quad\bsy\in U\,,
\end{align}
where the functions $\psi_j(\bsx)$ are given.  For example, the functions
$\psi_j$ may belong to the Karhunen-Lo\`eve eigensystem of a covariance
operator, or other suitable function systems in $L^2(D)$. We note that $n$
and $n_0$ represent the square of the non-zero physical refractive index,
and hence $n$ and $n_0$ are positive.

The paper~\cite{FLL15} studied a different computational scheme for wave
propagation in random media, using the standard sign-indefinite
formulation of the Helmholtz equation. In that paper the squared
refractive index $n$ was taken to be of the form
\begin{equation}\label{eq:FL_n}
 n(\bsx,\omega) 
 = [ 1 + \epsilon\,\eta(\bsx,\omega)]^2,
 \qquad \bsx \in D, \quad \omega \in \Omega,
 \end{equation}
with $\epsilon$ a small perturbation parameter controlling the magnitude
of the random fluctuation, $\eta$ being a random process satisfying the
constraint that $\mathbb{P} \left\{\omega \in \Omega:
\|\eta(\cdot,\omega)\|_{L^\infty(D)} \leq 1 \right\} = 1$, and $\Omega$
being a sample space. With $D$ being star-shaped with respect to the
origin, the authors established in~\cite[Theorem 2.15]{FLL15}
well-posedness of the continuous stochastic Helmholtz model, under the
restriction that the parameter $\epsilon$ is of the order $1/(kL)$. A
discrete form of the stochastic Helmholtz model was then developed
in~\cite{FLL15} by writing the stochastic solution as a series in powers
of $\epsilon^j, j = 0, 1, 2,\ldots$. The coefficients in the series
expansion were approximated in~\cite{FLL15} using the interior-penalty
discontinuous Galerkin (IPDG) discretization method in space, and Monte
Carlo (MC) cubature in the stochastic variables, the IPDG method being
chosen because of its unconditional stability. Because of the low-order
convergence of MC approximations, the approach in~\cite{FLL15} requires
substantially more sampling points in $\Omega$ (and hence more Helmholtz
system solves) compared to the higher-order QMC method and finite element
method (FEM) Galerkin scheme used in the present paper.

The general operator-theory approach in~\cite {DKLNS14} and the expansion
in powers of $\epsilon$ approach in~\cite{FLL15} both require that,
roughly speaking,
\[
kL\times  \; \mbox{(some norm of the stochastic variables) be not large}.
\]
(For the case of~\cite {DKLNS14} see Appendix \ref{append:pert}.)  We
shall see in \eqref{eq:cond} below that the same is true of the present
method. Thus all three approaches have this feature, but with the
difference that in \cite{DKLNS14} the requirement is absolute (see
\eqref{eq:neu-ass}), whereas in the present work the consequence of taking
larger values of $kL$ is only to increase the constants in our error
bounds.

Yet another approach to the Helmholtz problem with random refractive index
has been proposed recently in \cite{PS20}. There well-posedness and
stability of the sign-indefinite formulation of the stochastic continuous
problem has been proved. However, that article~\cite{PS20} does not
consider any form of numerical discretization. Other recent papers
concerned with the Helmholtz problem with variable coefficients
are~\cite{BCFG17,CF16,CFN20,GS20,GSW20}.

The main challenge in the present article lies in the design and numerical
analysis of a high-order QMC-FEM for the evaluation of expected values
(that is, of integrals with respect to $\bsy$ over a hypercube of
length~$1$) of linear functionals of the solution $u$. As with the earlier
applications of QMC-FEM to diffusion problems, the key is to find
computable bounds on appropriate mixed partial derivatives of $u$ with
respect to components of $\bsy$. The difference in this case is that
finding such bounds is now very much harder. The reason for the additional
difficulty lies in the much greater complexity of the coercive
formulation~\cite{GanMor20}. In particular, unlike the situation with the
diffusion problem, both the trial and test functions of the QMC-FEM
analysis have stochastic components.

More precisely, for each $\bsy \in U$, we seek a continuous wave field
solution in a special subspace $V$ of $H^1(D)$, see \eqref{HS} below. We
fix $u(\cdot, \bsy) \in V$ to be the unique solution of a sign-definite
weak formulation of \eqref{eq:pde_omega}--\eqref{eq:axy-unif}, and we
consider the quantity of interest (QoI) to be a bounded linear functional
$G\in V^*$ of $u(\cdot, \bsy)$, denoted in this article by $[G(u)](\bsy) =
G(u(\cdot,\bsy))$, where $V^*$ denotes the dual space of~$V$, with norm
given by \eqref{eq:dual-norm} below. An example of $G \in V^*$ is the
average wave field in the heterogenous medium: $G(u(\cdot,\bsy)) \,=\,
\int_D u(\bsx, \bsy) \,\rd\bsx$. The aim of this article is to design and
analyze efficient QMC-FEM algorithms to compute approximations to the
expected value of $G(u(\cdot,\bsy))$, expressed as an infinite-dimensional
integral over~$\bsy$:
\[
  \int_{[-\tfrac{1}{2},\tfrac{1}{2}]^\bbN} G(u(\cdot,\bsy))\,\rd\bsy.
\]
Key ingredients of our strategy are: (i) truncating the infinite series
in~\eqref{eq:axy-unif} to finitely many $s$ terms; (ii) discretizing the
solution in the spatial variable using FEM based on a mesh parameter~$h$;
and (iii) approximating the expected value integral by an $N$-point QMC
cubature rule.

For $kL \geq 1$, we prove that the combined error for the QMC-FEM
approximation is of the order
\[
  s^{-\frac{2}{p_0}+1}
  \;+\; kL\, h^{p}
  \;+\;
  \begin{cases}
  N^{-\min\big(\frac{1}{p_1} - \frac{1}{2}, 1-\delta\big)},\, \delta\in (0,\tfrac{1}{2}), & \mbox{for first order randomized QMC}, \\
  N^{-\frac{1}{p_1}}, & \mbox{for higher order deterministic QMC},
  \end{cases}
\]
where $p$ is the degree of the finite element spline basis functions
constructed using a tessellation of $D$ with mesh-width $h$, and $p_0,
p_1\in (0,1)$ satisfy the summability and wavenumber decay conditions
\begin{equation}\label{eq:cond}
  \sum_{j\ge 1} \Big(kL\,\|\psi_j\|_{L^\infty(D)}\Big)^{p_0} \le K_0
  \qquad\mbox{and}\qquad
  \sum_{j\ge 1}\Big(kL\, \|\psi_j\|_{W^{1,\infty}(D)}\Big)^{p_1} \le K_1,
\end{equation}
with $K_0, K_1\in\bbR$ independent of the wavenumber $k$. In particular,
the order constant in the error bound depends on $f$, $g$, $G$, but is
independent of $k$.

The rest of this article is organized as follows. In
Section~\ref{sec:sign-def}, for each fixed $\bsy \in U$, we introduce the
coercive formulation of the stochastic model, and recall from
\cite{GanMor20} a wavenumber-explicit spatial regularity bound on the
unique solution. In Section~\ref{sec:overview} we provide an overview of
the analysis needed to obtain the final combined error bound. In
Section~\ref{sec:param_der} we derive explicit bounds on partial
derivatives with respect to components of $\bsy$ of the solution $u$, as
needed for the QMC analysis and the construction of QMC points. In
Section~\ref{sec:dim_trunc} we quantify the effect of truncation of the
infinite series for $n(\bsx,\bsy)$. In Section~\ref{sec:fem} we describe
the error associated with high-order FEM discretization. In
Section~\ref{sec:qmc} we focus on the efficient choice of the randomized
and deterministic QMC quadrature rules. In Appendix~A we describe the
alternative small perturbation QMC-FEM approach. In Appendix~B we prove a
technical lemma.

\section{A coercive reformulation of the stochastic Helmholtz model}\label{sec:sign-def}

A coercive variational formulation  was developed and analyzed  recently
in~\cite{GanMor20} for a deterministic wave propagation model with an
inhomogeneous absorbing boundary condition.  Here we extend the method to
our stochastic model.

The first step is to recognize that given data $f \in L^2(D), g\in
L^2(\partial D)$, for each fixed $\bsy \in U$, any sufficiently regular
solution $u(\cdot, \bsy) \in H^1(D)$ of our model boundary value problem
(BVP)
\begin{align}\label{eq:pde_y}
 \left[\Delta +k^2\,n(\bsx, \bsy)\right]u(\bsx, \bsy) &\,=\, -f(\bsx)\,,
 \quad \bsx \in  D\,,
 \quad  \text{and} \quad \\
 \frac{\partial u}{\partial \bfn}(\widetilde{\bsx}, \bsy)-\ii \,k\,u(\widetilde{\bsx} , \bsy) &\,=\, g(\widetilde{\bsx} )\,,
 \quad   \widetilde{\bsx} \in \partial D\,, \nonumber
\end{align}
has three additional smoothness properties: (i) $\Delta u(\cdot, \bsy) \in
L^2(D)$; (ii) $\frac{\partial u}{\partial \bfn}(\cdot, \bsy) \in
L^2(\partial D)$; and (iii) $u(\cdot, \bsy) \in H^1(\partial D)$. The
first two properties follow directly from~\eqref{eq:pde_y} and the third
property follows from the fact that $\nabla u(\cdot, \bsy) \in
L^2(\partial D)$, since $\nabla u(\cdot, \bsy) = \bfn\, \frac{\partial
u}{\partial \bfn}(\cdot, \bsy) + \nabla_{\partial D} u(\cdot, \bsy)$ and
we have the regularity result from~\cite[Theorem 4.2]{McLean00} that
surface gradient $\nabla_{\partial D} u(\cdot, \bsy) \in L^2(\partial D)$.
We incorporate such natural smoothness properties of the Helmholtz PDE
model~\eqref{eq:pde_y} in the following Hilbert space:
\begin{equation}
\label{HS}
 V \,:=\, \Big\{
 w \;:\; w \in H^1(D)\,, \; \Delta w \in L^2(D)\,, \; w \in H^1 (\partial D)\,,\;
 \frac{\partial w}{\partial \bfn} \in L^2(\partial D) \Big\}\,.
\end{equation}
Following~\cite{GanMor17, GanMor20}, for the stochastic heterogenous model
we equip $V$ with the following norm
\begin{align} \label{NORM}
 \|w\|^2_V &\,:=\,
 k^2\,\|w\|^2_{L^2(D)} + \|\nabla w\|^2_{L^2(D)} + \frac{1}{k^2}\,\|\Delta w\|^2_{L^2(D)} \nonumber \\
 &\qquad
 + L\, \Big( k^2\,\|w\|^2_{L^2(\partial D)} + \|\nabla_{\partial D}w\|^2_{L^2(\partial D)}
 + \Big\|\frac{\partial w}{\partial \bfn}\Big\|^2_{L^2(\partial D)} \Big)\,,
\end{align}
where $L$ is a characteristic length of the Lipschitz domain $D\subset
\mathbb{R}^d,~d=2,3$. Note that each term in \eqref{NORM} scales in the
same way under a change of length scale. Throughout this article, when
considering the trace of a function $w \in H^{s}(D)$ as a function in
$H^{s-1/2}(\partial D)$, for notational convenience we drop the Dirichlet
trace operator $\gamma$. (That is, we drop $\gamma$ and write $w$ instead
of $\gamma\, w$ whenever it is considered as a function on $\partial D$.)

For each fixed $\bsy \in U$, to prove the unique solvability of the
BVP~\eqref{eq:pde_y} we ensure the coercivity property of the variational
formulation by assuming the following three conditions on the geometry and
medium of the wave propagation:
\begin{enumerate}
\item[(A0)] \label{A0} %
The domain $D \subset \mathbb{R}^d$, for $d=2,3$, with diameter $L$,
is star-shaped with respect to a ball centered at the origin. That is,
there exist constants $\widehat{\gamma}$, $\widehat{\mu}$ with
$0<\widehat{\gamma} \leq \widehat{\mu} \leq 1$ such that
\begin{equation*} 
 \widehat{\gamma}\, L \,\leq\, \widetilde\bsx\cdot \bfn(\widetilde\bsx) \,\leq\, \widehat{\mu} L\,,
 \qquad \widetilde {\bsx} \in \partial D\,.
\end{equation*}
We now fix $L$ by defining $L:=\sup_{\bsx\in D}\|\bsx\|$, where
$\|\bsx\|$ is the Euclidean norm of $\bsx$.
\item[(A1)]\label{A1}%
For $\bsy\in U$ and $\bsx\in D$, there exist constants $n_{\max},
n_{\min}, b_{\max}$ and $b_{\min}$ such that almost everywhere
\begin{equation}\label{eq:coer_assump}
 0 \,<\, n_{\min} \,\le\, n(\bsx,\bsy) \,\le\, n_{\max}\,,
 \end{equation}
\begin{equation}\label{eq:weak-nontrap}
 0 \,<\, b_{\min} \,\le\, \nabla \cdot (\bsx\, n(\bsx,\bsy)) \,\le\, b_{\max}\,,
 \qquad b_{\min} \,>\, (d-2)\,n_{\max}\,.
 \end{equation}
 \item[(A2)] \label{A2}%
The mean field and perturbation functions satisfy $n_0\in
W^{1,\infty}(D)$, $\psi_j \in W^{1,\infty}(D)$ and $\sum_{j\geq
1}\|\psi_j\|_{W^{1,\infty}(D)} < \infty$, where throughout the article
\[
  \|w\|_{W^{1,\infty}(D)}
  := \max\big\{\|w\|_{L^\infty(D)},\;L\, \|\nabla w\|_{L^\infty(D)}\big\}.
\]
\end{enumerate}

The positivity and boundedness of the refractive index in
\eqref{eq:coer_assump} is well known for all practical heterogeneous wave
propagation media. As described in detail in~\cite[Remark 2.1]{GanMor20},
the two inequalities on $b_{\min}$ in \eqref{eq:weak-nontrap} are
necessary to ensure the physical constraint that the (geometric-optical)
rays are \emph{non-trapping}~\cite[Page 191]{EgoShu93}. For a detailed
geometric interpretation related to the positivity condition
in~\eqref{eq:weak-nontrap}, see~\cite[Section~7]{GPS19}.

To develop the sign-definite variational formulation of the BVP, we
consider the following operators~\cite{GanMor17,GanMor20}
\begin{equation}
 \label{MU}
 \M_\ell\, w \;:=\;
 \bsx\cdot\nabla w \,-\, \ii \,k\,L\,\widehat{\beta_\ell}\, w \,
 +\, \alpha_\ell\, w, \quad \ell=1,2\,.
\end{equation}
The four parameters $\alpha_1, \alpha_2, \widehat{\beta_1},
\widehat{\beta_2} \in \mathbb{R}$  and an additional parameter $A \in
\mathbb{R}$  will subsequently play a crucial role. To explain the
notation, the three parameters without the ``hat'' tag are independent of
the geometry, while the parameters tagged with a ``hat'' will occur in
this article in combination with the ``acoustic size'' $k\,L$.

Next, for each fixed $\bsy \in U$ and for $f\in L^2(D)$ and $g\in
L^2(\partial D)$, with $\La$ and  $n(\bsx,\bsy)$ given by
\eqref{eq:pdeoper_omega}--\eqref{eq:axy-unif}, we recall a sesquilinear
form $\bilin_\bsy :V\times V \to \mathbb{C}$ and an antilinear functional
$\antilin_{\bsy,f,g} =  \antilin_{\bsy} :V\to\mathbb{C}$, introduced
in~\cite{GanMor20}:
\begin{align}
\label{VSF}
 \bilin_\bsy(v,w) &\,:=\,
 \int_{D} \Big[
 \Big(\M_2 v+\frac{A}{k^2}\La v\Big)\overline{\La w}
 \,+\, \big(2-d+\alpha_1+\alpha_2 + \ii\,k\,L(\widehat\beta_1-\widehat\beta_2)\big)\nabla v \cdot \overline{\nabla w}
 \nonumber \\
 &\qquad\qquad\qquad\qquad +
 \big(-\alpha_1-\alpha_2-\ii\,k\,L(\widehat\beta_1-\widehat\beta_2)\big)\,
 k^2\,n\,v\,\overline{w}
 +k^2\,\big(\nabla\cdot(\bsx n)\big)\,v\,\overline{w}\, \Big]
 \,\rd\bsx \nonumber\\
 & \qquad
 -\int_{\partial D}
 \Big[\overline{\M_1 w}\, \ii\,k\,v
 + \big(\bsx\cdot\nabla_{\partial D}v - \ii\,k\,L\,\widehat\beta_2\,v \,
 +\, \alpha_2\, v\big)
 \overline{\frac{\partial w}{\partial \bfn}} \nonumber\\
 &\qquad\qquad\qquad\qquad\qquad\qquad
 +(\bsx\cdot\bfn)
  \big(k^2\,n\,v\,\overline{w} -\nabla_{\partial D} v \cdot \overline{\nabla_{\partial D} w}\big)
  \Big]\,\rd S\,,
\end{align}
and
\begin{equation} \label{VF}
 \antilin_\bsy(w) \,:=\,
 \int_{D}\Big(\overline{\M_1 w}-\frac{A}{k^2}\overline{\mathcal{L} w}\Big)f \,\rd\bsx
 + \int_{\partial D}\overline{\M_1 w}\,g\,\rd S\,.
\end{equation}

Using the technical details in the proof of~\cite[Section~2]{GanMor20}, we
have the following consistency result connecting the PDE model and the
variational formulation determined by the above sesquilinear form and
antilinear functional: For each $\bsy \in U$, if $u(\cdot, \bsy) \in
H^1(D)$ solves the wave propagation PDE model~\eqref{eq:pde_y}, then $u
(\cdot, \bsy) \in V$ satisfies the variational equation
\begin{equation}\label{eq:ses_var}
 \bilin_\bsy(u,w) \,=\, \antilin_\bsy(w) \qquad \text{for all}\quad  w\in V\,.
\end{equation}

The following coercivity, continuity, and unique solvability of
~\eqref{eq:ses_var}  with wavenumber-explicit bounds follow from  similar
results proved in~\cite{GanMor20}. In particular, for acoustic size $kL
\geq 1$, the $V$-norm spatial regularity bound of the unique solution of
the wave propagation model is independent of the wavenumber. Such
wavenumber-explicit bounds play a crucial role in the analysis and
construction of QMC approximations. Below we use the standard norm for the
dual $V^*$ of~$V$:
\begin{equation} \label{eq:dual-norm}
 \|G\|_{V^*} \,:=\,
 \sup\left\{ \frac{|G(w)|}{\|w\|_V} \,:\,w\in V\,, w\neq 0\right\}.
\end{equation}

\begin{theorem}[{\cite[Theorems~2.1, 3.1, 3.2, and~4.1]{GanMor20}}]
 \label{COERTHM}
Let the assumptions \textnormal{(A0)} and \textnormal{(A1)} hold. If the
three parameters $A, \alpha_1, \widehat{\beta_1}$ are chosen such that
\begin{align} \label{eq:restrictions}
 \frac{d-2}{2} \,<\,\alpha_1\,<\,\frac{b_{\min}}{2\,n_{\max}}\,, \quad
 0\,<A\,<\frac{b_{\min}-2\,\alpha_1\,n_{\max}}{2\,n_{\max}^2}\,, \quad
 \widehat{\beta_1} \,\geq\, \frac{n_{\max}\,\widehat{\mu}}{2}
 + \frac{2\,\widehat{\mu}^2}{\gamma} + \frac{\widehat{\gamma}}{2},
\end{align}
then for all $\bsy\in U$, $f\in L^2(D)$ and $g\in L^2(\partial D)$ we have
\begin{align}
 \operatorname{Re} [\bilin_\bsy(w,w)] \label{eq:coer_res}
 &\,\geq\, C_{\rm coer}\, \|w\|_V^2 && \text{for all}\quad w \in V\,,
 \\
 |\bilin_\bsy(v,w)| \label{eq:cont}
 &\,\leq\, C_{\rm cont}(kL)\, \|v\|_V\, \|w\|_V &&  \text{for all}\quad v,w \in V\,,
 \\
 \|\antilin_{\bsy,f,g}\|_{V^*} \label{eq:G-bd}
  &\,\leq\, C_{\rm func}(kL)\, \big(L\, \|f\|_{L^2(D)} + L^{1/2}\,\|g\|_{L^2(\partial D)} \big),
\end{align}
with
\begin{align*}
 C_{\rm coer} 
 &\,:=\,
 \frac{1}{2}\min \left\{  2-d+2\alpha_1\,,\, b_{\min}-2\,\alpha_1\,n_{\max}-2\,A\,n_{\max}^2\,,\,
 A,\,\frac{\widehat{\gamma}}{2}\right\},
 \\
 C_{\rm cont}(kL) 
 &\,:=\, \sqrt{3} \max \left\{
 \begin{array}{l}
 |2-d+\alpha_1+\alpha_2|+kL\,|\widehat\beta_1-\widehat\beta_2|\,,\,  \\
 A\,n_{\max}+|\alpha_2-\ii\,kL\,\widehat\beta_2| + kL + A\,, \\
 \displaystyle\frac{\alpha_1}{kL} + \widehat\beta_1 + n_{\max}\,\hat\mu\,,\quad
 \displaystyle\frac{|\alpha_2|}{kL} + |\widehat\beta_2| + 2\,\hat\mu\,,\quad\,2\,, \vspace{0.1cm} \\
 \big(|\alpha_1+\alpha_2|+ b_{\max}+kL\,|\widehat\beta_1-\widehat\beta_2|\big) n_{\max} \\
 \quad + \big(A\,n_{\max}^2+n_{\max}|\alpha_2-\ii\,kL\,\widehat\beta_2|\big) + kL\,n_{\max}+A\,n_{\max}
 \end{array}
 \right\},
 \\
 C_{\rm func}(kL) 
 &\,:=\, \sqrt{3}\max\left\{
 1\,,\, \frac{A}{kL}\,,\, \frac{\alpha_1+A\,n_{\max}}{kL} + \widehat\beta_1
 \right\}.
\end{align*}
The coercivity constant $C_{\rm coer}$ is independent of the wavenumber.
The continuity constant satisfies $C_{\rm cont}(kL) = \calO\big(kL +
(kL)^{-1}\big)$. The functional constant satisfies $C_{\rm func}(kL) =
\calO\big(1 + (kL)^{-1}\big)$, and so is bounded independently of the
wavenumber if $kL \ge 1$.

Consequently, for each $\bsy \in U$, the variational
formulation~\eqref{eq:ses_var} has a unique solution $u(\cdot, \bsy) \in
V$ and  satisfied the regularity bound
\begin{equation}\label{eq:ureg}
 \|u(\cdot, \bsy)\|_V \,\leq\,
 \frac{C_{\rm func}(kL)}{C_{\rm coer}}
 \big(L\, \|f\|_{L^2(D)} + L^{1/2}\,\|g\|_{L^2(\partial D)} \big)
 \qquad \text{for all} \quad \bsy \in U\,,
\end{equation}
which is bounded independently of the wavenumber if $kL \ge 1$.
\end{theorem}

We note that \eqref{eq:cont} and \eqref{eq:G-bd} hold even without the
weak non-trapping condition \eqref{eq:weak-nontrap}. In particular, as
described in~\cite{GanMor20}, only the proof of coercivity requires all
assumptions mentioned in Theorem~\ref{COERTHM}.

\section{Overview of our method and error analysis} \label{sec:overview}

The main aim of this article is to design and analyze efficient algorithms
to compute approximations to the expected value of $G(u(\cdot,\bsy))$,
expressed as an infinite-dimensional integral over~$\bsy$:
\begin{align} \label{eq:int-unif}
  I(G(u))
  &\,:=\,
  \int_{[-\tfrac{1}{2},\tfrac{1}{2}]^\bbN} G(u(\cdot,\bsy))\,\rd\bsy \,:=\,  \lim_{s\to\infty}I_s(G(u)),
  \end{align}
  with
   \begin{align} \label{eq:int-unifs}
  &I_s(G(u))\,:=\, \int_{[-\tfrac{1}{2},\tfrac{1}{2}]^s}
        G(u(\cdot,(y_1,\ldots,y_s,0,0,\ldots)))\,\rd y_1\cdots\rd y_s\,.
\end{align}

Key ingredients of our strategy are: (i) truncating the infinite series
in~\eqref{eq:axy-unif} to finitely many $s$ terms, yielding the
dimensionally-truncated solution $u_s$; (ii) discretizing $u_s$ in the
spatial variable using FEM based on a mesh parameter $h$, leading to the
discrete solution $u_{s,h}$; and (iii) approximating the $s$-dimensional
expected value integral of $G(u_{s,h})$ by an $N$-point QMC cubature rule
$Q_{s,N}$. The precise details regarding $u_s$, $u_{s,h}$ and the QMC rule
$Q_{s,N}$ are given in later sections. For now it suffices to say that we
can write the combined error using the triangle inequality as a sum of
three terms: the \emph{dimension truncation error}, the \emph{FEM
discretization error}, and the \emph{QMC cubature error}:
\begin{align*} 
 &|I(G(u)) -  Q_{s,N} (G(u_{s,h})) | \nonumber \\
 &\,\le\, |(I - I_s)(G(u))| \,+\,  |I_s(G(u_s - u_{s,h}))| \,+\, | I_s(G(u_{s,h})) - Q_{s,N}(G(u_{s,h}))|.
\end{align*}
Alternatively, if the QMC rule is randomized then we have the mean-square
error
\begin{align*} 
 & \bbE_{\rm rqmc} \Big[ |I(G(u)) -  Q_{s,N} (G(u_{s,h});\cdot) |^2 \Big] \nonumber \\
 &\,\le\, 2\, |(I - I_s)(G(u))|^2 \,+\,  2\, |I_s(G(u_s - u_{s,h}))|^2 \,+\,
 \bbE_{\rm rqmc} \Big[| I_s(G(u_{s,h})) - Q_{s,N}(G(u_{s,h});\cdot)|^2 \Big],
\end{align*}
where the expectation $\bbE_{\rm rqmc}$ is taken with respect to the
random element in the QMC rule (see Section~\ref{sec:qmc}).


\section{Stochastic parameter regularity of random wave field}\label{sec:param_der}

For the error analysis it is crucial to understand the behavior of
multi-index high-order derivatives of the solution of~\eqref{eq:ses_var}
with respect to the stochastic variables $y_j,~j\ge 1$. To this  end,   we
first introduce some notation. For a multi-index $\bsnu = (\nu_j)_{j\ge1}$
with $\nu_j\in \{0,1,2,\ldots\}$, we write its ``order'' as $|\bsnu| :=
\sum_{j\ge 1} \nu_j$ and its ``support'' as $\supp(\bsnu) := \{j\ge 1:
\nu_j\ge 1\}$. Furthermore, we write $\bsnu! := \prod_{j\ge 1} (\nu_j!)$,
which is different from $|\bsnu|! = (\sum_{j\ge 1} \nu_j)!$. We denote by
$\indx$ the (countable) set of all ``finitely supported'' multi-indices: $
  \indx
  \,:=\,
  \{ \bsnu \in \bbN_0^\bbN : \supp(\bsnu) < \infty \}
  .
$
For $\bsnu\in\indx$, we denote the $\bsnu$-th partial derivative with
respect to the parametric variables $\bsy$ by
\begin{align*}
  \partial^{\bsnu} \,=\, \partial^{\bsnu}_{\bsy}
  \,=\, \frac{\partial^{|\bsnu|}}{\partial y_1^{\nu_1}\partial y_2^{\nu_2}\cdots} \,.
\end{align*}

For any sequence of real numbers $\bsb = (b_j)_{j\ge 1}$, we write
$\bsb^\bsnu := \prod_{j\ge 1} b_j^{\nu_j}$. By $\bsm\le\bsnu$ we mean that
the multi-index $\bsm$ satisfies $m_j\le \nu_j$ for all $j$. Moreover,
$\bsnu-\bsm$ denotes a multi-index with the elements $\nu_j-m_j$, and
$\binom{\bsnu}{\bsm} := \prod_{j\ge 1} \binom{\nu_j}{m_j}$. We denote by
$\bse_j$ the multi-index whose $j$th component is $1$ and whose other
components are $0$. We will make repeated use of the Leibniz product rule
\begin{align} \label{eq:leibniz}
  \partial^\bsnu (PQ) \,=\,
  \sum_{\bsm\le\bsnu} \binom{\bsnu}{\bsm} (\partial^{\bsm} P)\, (\partial^{\bsnu-\bsm} Q)\,.
\end{align}

For a general multi-index derivative, $\partial^{\bsnu}$, we obtain the
following result.

\begin{lemma} \label{lemma:rec1}
Let the assumptions and parameter restrictions in Theorem~\ref{COERTHM}
hold. For each $\bsy \in U$  let $u(\cdot, \bsy) \in V$ be the unique
solution of~\eqref{eq:ses_var}. Then for any $\bsnu\in\indx$ (including
$\bsnu=\bszero$) and any $u,w,z\in V$,
\begin{equation}\label{eq:der_rec}
 \bilin_\bsy(\partial^\bsnu u,w) \,=\,
 \sum_{j\in\supp(\bsnu)} \nu_j\, R_j(\partial^{\bsnu-\bse_j} u,w) + S_{\bsnu}(u,w) +  T_\bsnu(w),
\end{equation}
where
\begin{align}\label{eq:Rj}
 R_j(z, w)
 &\,:=\,
 - \int_{D} \Big[  A\,\psi_j\, z\, \overline{\La w} +
 \Big(\M_2 z+\frac{A}{k^2}\La z\Big)\,k^2\,\psi_j\,\overline{w}
 \nonumber\\
 &\qquad\qquad
 + \big(-\alpha_1-\alpha_2-\mathrm{i}\,k\,L\,(\widehat\beta_1-\widehat\beta_2)\big)\,k^2\,\psi_j\,z\,\overline{w}
 + k^2\,\big(\nabla\cdot(\bsx \psi_j) \big)\,z\,\overline{w}\Big]\,\rd\bsx
 \nonumber\\
 &\qquad
 + k^2 \int_{\partial D} (\bsx\cdot \bfn) \psi_j\,z\,\overline{w}\,\rd S,
\end{align}
\begin{align} \label{eq:Snu}
 &S_{\bsnu}(u,w) \nonumber\\ &\,:=\,
 \begin{cases}
  0 & \mbox{if}\quad |\bsnu| = 0, 1, \\
 -  A\, k^2 \displaystyle\sum_{j\in\supp(\bsnu)} \sum_{\ell\in\supp(\bsnu-\bse_j)}
 \nu_j (\bsnu-\bse_j)_\ell \int_D  \psi_j\,\psi_\ell\, (\partial^{\bsnu-\bse_j-\bse_\ell}u) \, \overline{w}
  & \mbox{otherwise},
  \end{cases}
  \end{align}
and
\begin{align} \label{eq:Fnu}
 T_\bsnu(w)
 &\,:=\,
  \begin{cases}
  \antilin_\bsy(w)  & \mbox{if} \quad \bsnu = \bszero, \\
  \displaystyle-A\int_D \psi_j\, \overline{w}\, f\, \, \rd\bsx & \mbox{if} \quad \bsnu = \bse_j, \\
  0 & \mbox{otherwise}.
  \end{cases}
\end{align}
\end{lemma}

\begin{proof}
For any $\bsy\in U$, let $u(\cdot, \bsy) \in V$ be the unique solution
of~\eqref{eq:ses_var}. For any $w \in V$ (independent of~$\bsy$) and any
$\bsnu\in\indx$, we will prove the lemma by differentiating and equating
the two sides of \eqref{eq:ses_var}, that is,
\begin{equation}\label{eq:diff}
  \partial^\bsnu (\bilin_\bsy(u(\cdot,\bsy),w)) \,=\, \partial^\bsnu (\antilin_\bsy(w))\,.
\end{equation}

Starting with the left-hand side of \eqref{eq:diff}, we note from the
sesquilinear form \eqref{VSF} that the factors which depend on $\bsy$ are
$u(\bsx,\bsy)$ and $n(\bsx,\bsy)$ as well as (cf.\
\eqref{eq:pdeoper_omega})
\begin{equation} \notag
  (\La u)(\bsx,\bsy) \,=\, \Delta u(\bsx,\bsy) + k^2\,n(\bsx,\bsy)\, u(\bsx,\bsy)
  \quad\mbox{and}\quad
  (\La w)(\bsx) \,=\, \Delta w(\bsx) + k^2\,n(\bsx,\bsy)\, w(\bsx)\,.
\end{equation}
Using the definition of $n(\bsx,\bsy)$ in \eqref{eq:axy-unif}, we have
\begin{align} \label{eq:diff-a}
  \partial^{\bsm} n(\bsx,\bsy)
  \,=\,
  \begin{cases}
  n(\bsx,\bsy) & \mbox{if } \bsm = \bszero, \\
  \psi_j(\bsx) & \mbox{if } \bsm = \bse_j, \\
  0 & \mbox{otherwise},
  \end{cases}
\end{align}
It follows that (suppressing from here on the dependence on $\bsx$ and
$\bsy$)
\begin{align} \label{eq:diff-law}
  \partial^\bsm (\La w)
  &\,=\,
  \begin{cases}
  \La w & \mbox{if } \bsm = \bszero, \\
  k^2\,\psi_j\, w & \mbox{if } \bsm = \bse_j, \\
  0 & \mbox{otherwise},
  \end{cases}
\end{align}
and using \eqref{eq:leibniz} we obtain
\begin{align} \label{eq:diff-lau}
  \partial^\bsnu (\La u)
  &\,=\, \Delta (\partial^\bsnu u)
  + k^2\, \sum_{\bsm\le\bsnu} \binom{\bsnu}{\bsm} (\partial^\bsm n)\, (\partial^{\bsnu-\bsm} u) \nonumber\\
  &\,=\, \Delta (\partial^\bsnu u) + k^2\,n\,(\partial^\bsnu u)
  + k^2\,  \sum_{j\in\supp(\bsnu)}  \nu_j\,\psi_j\, (\partial^{\bsnu-\bse_j} u) \nonumber\\
  &\,=\, \La (\partial^\bsnu u)
  + k^2\, \sum_{j\in\supp(\bsnu)}  \nu_j\,\psi_j\, (\partial^{\bsnu-\bse_j} u)\,.
\end{align}

To ease our derivation below, we split the sesquilinear form \eqref{VSF}
into three terms, $\bilin_\bsy(u,w) = \bilin_1(u,w) + \bilin_2(u,w) +
\bilin_3(u,w)$, based on the level of dependency on $\bsy$:
\begin{align*}
 \bilin_1(u,w) &\,:=\,
 \int_{D} \Big(\M_2u+\frac{A}{k^2}\La u\Big)\overline{\La w}\,\rd\bsx
 \\
 \bilin_2(u,w) &\,:=\,
 \int_{D} k^2 \big( \xi_2\,n + \nabla\cdot(\bsx\, n)\big)\, u\,\overline{w}\,\rd\bsx
 -\int_{\partial D} k^2\,(\bsx\cdot\bfn)\, n\,u\,\overline{w}\,\rd S\,,
 \\
 \bilin_3(u,w) &\,:=\,
 \int_{D} \xi_1\,\nabla u \cdot \overline{\nabla w}
 \,\rd\bsx \\
 &\qquad
 -\int_{\partial D}
 \Big(\overline{\M_1 w}\, \ii\,k\,u
 + \big(\bsx\cdot\nabla_{\partial D}u +\xi_3\, u\big)
 \overline{\frac{\partial w}{\partial \bfn}}
 - (\bsx\cdot\bfn)
  \nabla_{\partial D} u \cdot \overline{\nabla_{\partial D} w}\Big)\,\rd S\,,
\end{align*}
with the abbreviations $\xi_1 := 2-d+\alpha_1+\alpha_2+\mathrm{i}\,
kL(\widehat\beta_1-\widehat\beta_2)$, $\xi_2 :=
-\alpha_1-\alpha_2-\mathrm{i}\, kL(\widehat\beta_1-\widehat\beta_2)$, and
$\xi_3 := - \mathrm{i}\,kL\widehat\beta_2 + \alpha_2$.

It is easy to see that
\[
 \partial^\bsnu (\bilin_3(u,w)) \,=\, \bilin_3(\partial^\bsnu u,w)\,.
\]
Using \eqref{eq:leibniz} and \eqref{eq:diff-a} we obtain
\begin{align*}
  &\partial^\bsnu (\bilin_2(u,w))
  \,=\, \sum_{\bsm\le\bsnu} \binom{\bsnu}{\bsm} \bigg[
  \int_D k^2 \big(
  \xi_2\, (\partial^\bsm n) +  \nabla \cdot (\bsx\, (\partial^{\bsm} n))\big)
  (\partial^{\bsnu-\bsm} u) \,\overline{w}\,\rd\bsx \\
  &\qquad\qquad\qquad\qquad\qquad\qquad - \int_{\partial D}
  k^2 (\bsx\cdot\bfn) (\partial^\bsm n) (\partial^{\bsnu-\bsm} u)\,\overline{w}
  \,\rd S\bigg] \\
  &=\, \bilin_2(\partial^\bsnu u,w) \\
  &\quad
  + \sum_{j\in\supp(\bsnu)}\!\!\! \nu_j\, \bigg[
  \int_D k^2 (\xi_2\,\psi_j + \nabla \cdot (\bsx\, \psi_j))\, (\partial^{\bsnu-\bse_j} u)\,\overline{w}\,\rd\bsx
  - \int_{\partial D} k^2 (\bsx\cdot\bfn) \psi_j (\partial^{\bsnu-\bse_j} u)\,\overline{w}\,\rd S \bigg]\,.
\end{align*}
Using \eqref{eq:leibniz} and \eqref{eq:diff-law}, followed by applying
\eqref{eq:diff-lau} with $\bsnu$ replaced by $\bsnu-\bse_j$ and index $j$
replaced by~$\ell$, we obtain
\begin{align*}
  &\partial^\bsnu (\bilin_1(u,w))
  \,=\,
  \sum_{\bsm\le\bsnu} \binom{\bsnu}{\bsm} \int_D
   \Big[\partial^{\bsnu-\bsm} \Big(\M_2 u + \frac{A}{k^2} \La u\Big)\Big]
   \Big[\partial^{\bsm} (\overline{\La w})\Big] \,\rd\bsx \\
  &\,=\,
  \int_D
  \Big[\partial^{\bsnu} \Big(\M_2 u + \frac{A}{k^2} \La u\Big)\Big]\,
  \overline{\La w}\,\rd\bsx
  + \sum_{j\in\supp(\bsnu)}  \nu_j \int_D
  \Big[\partial^{\bsnu-\bse_j} \Big(\M_2 u + \frac{A}{k^2} \La u\Big)\Big]\,
  k^2\,\psi_j\, \overline{w} \,\rd\bsx\\
  &\,=\,
  \bilin_1(\partial^\bsnu u,w)
  + \int_D \frac{A}{k^2} k^2\, \sum_{j\in\supp(\bsnu)}  \nu_j\,
  \psi_j\, (\partial^{\bsnu-\bse_j} u)\,\overline{\La w}\,\rd\bsx \\
  &\qquad
  + \sum_{j\in\supp(\bsnu)}  \nu_j  \int_D \bigg[  \M_2 (\partial^{\bsnu-\bse_j} u)
  + \frac{A}{k^2} \La (\partial^{\bsnu-\bse_j} u) \\
  &\qquad\qquad\qquad\qquad\qquad
  + \frac{A}{k^2} k^2\, \sum_{\ell\in\supp(\bsnu-\bse_j)}  (\bsnu-\bse_j)_\ell\,\psi_\ell\, (\partial^{\bsnu-\bse_j-\bse_\ell} u)
  \bigg]\, k^2\,\psi_j\, \overline{w}\,\rd\bsx
  \\
  &\,=\,
  \bilin_1(\partial^\bsnu u,w) \\
  &\qquad
  + \sum_{j\in\supp(\bsnu)}\!\!\!  \nu_j \int_D \Big[
  A\, \psi_j\, (\partial^{\bsnu-\bse_j} u)\,\overline{\La w}
  + \Big(\M_2 (\partial^{\bsnu-\bse_j} u)
  + \frac{A}{k^2} \La (\partial^{\bsnu-\bse_j} u)\Big) k^2\,\psi_j\,\overline{w} \Big]\,\rd\bsx \\
  &\qquad
  + A\,k^2 \sum_{j\in\supp(\bsnu)} \sum_{\ell\in\supp(\bsnu-\bse_j)}\!\!\!  \nu_j\,(\bsnu-\bse_j)_\ell
  \int_D \psi_j\,\psi_\ell\, (\partial^{\bsnu-\bse_j-\bse_\ell} u)\,\overline{w}\,\rd\bsx\,.
\end{align*}

We note that for $|\bsnu|=1$ and $j \ge 1$, the set $\supp(\bsnu-\bse_j)$
is empty; in this case we take $\partial^{\bsnu-\bse_j-\bse_\ell}$ to be
the zero operator. Thus using~\eqref{eq:Rj} and \eqref{eq:Snu} we obtain,
\begin{align} \label{eq:final-der}
  \partial^\bsnu (\bilin_\bsy(u,w))
  \,=\, \bilin_\bsy(\partial^\bsnu u,w) - \sum_{j\in\supp(\bsnu)} \nu_j\, R_j(\partial^{\bsnu-\bse_j} u,w)
  - S_{\bsnu}(u,w)\,.
\end{align}

Now for the right-hand side of \eqref{eq:diff} we use \eqref{VF},
\eqref{eq:diff-a} and \eqref{eq:Fnu} to obtain
\begin{align} \label{eq:final-der2}
\partial^{\bsnu} (\antilin_\bsy w) \,=\, T_\bsnu(w)\,.
\end{align}
The required result is obtained by equating \eqref{eq:final-der} and
\eqref{eq:final-der2}.
\end{proof}

Next we derive a  bound on the parametric derivatives of the solution  of~\eqref{eq:ses_var} in the $V$-norm.

\begin{theorem} \label{theorem:bound}
Let the assumptions and parameter restrictions in Theorem~\ref{COERTHM}
hold, and assume additionally that \textnormal{(A2)} holds. For each $\bsy
\in U$, let $u(\cdot, \bsy) \in V$ be the unique solution
of~\eqref{eq:ses_var}. Then for all $\bsnu\in\indx$ (including
$\bsnu=\bszero$),
 \begin{equation}\label{eq:der-bd}
 \|\partial^{\bsnu}u(\cdot,\bsy)\|_{V}
 \,\le\, \frac{C_{\rm func}(kL)}{C_{\rm coer}} \big(L\, \|f\|_{L^2(D)}+L^{1/2} \|g\|_{L^2(\partial D)} \big)\,
 |\bsnu|!\,\bsUpsilon^\bsnu\;,
\end{equation}
where
\begin{align*}
 \bsUpsilon^\bsnu := \prod_{j\ge 1} \Upsilon_j^{\nu_j}\,,\qquad
  \Upsilon_j \,:=\,  C_{\rm regu}(kL)\, \|\psi_j\|_{W^{1,\infty}(D)},
\end{align*}
\begin{align}
 C_{\rm regu}(kL) \label{eq:c-regu}
 &\,:=\, \max\left\{\frac{C_R(kL)}{C_{\rm coer}} + \frac{A}{kL\,C_{\rm func}(kL)}\,,\,
  \frac{2\,C_R(kL)}{C_{\rm coer}}\,,\,
  \sqrt{\frac{2A}{C_{\rm coer}}} \right\},
\end{align}
\begin{align}
 C_R(kL) \label{eq:CR}
 &\,:=\, 2A(1+n_{\max})+ kL\big(1+\widehat\beta_2\big)  + |\alpha_2| +
 \big|\!-\alpha_1-\alpha_2-\mathrm{i}\,kL(\widehat\beta_1-\widehat\beta_2)\big| + d+1 + \hat\mu\,.
\end{align}
We have $C_R(kL) = \calO(kL+1)$, $C_{\rm func}(kL) = \calO\big(1 +
(kL)^{-1}\big)$, so $C_{\rm regu}(kL) = \calO\big(kL+(kL)^{-1}\big)$.
\end{theorem}

\begin{proof}
For the $\bsnu=\bszero$ case,~\eqref{eq:der-bd} follows from
\eqref{eq:ureg}. Let $|\bsnu| \geq 1$. We recall~\eqref{eq:der_rec} and
bound each term on the RHS of~\eqref{eq:der_rec}. Using the definition of
the $V$-norm in~\eqref{NORM}, for any $w \in V$ we have
\begin{equation*} 
 \|w \|_{L^2(D)} \leq \frac{1}{k} \| w \|_V, \;
 \|\nabla w \|_{L^2(D)} \leq  \|w \|_V, \;
 \|\Delta w \|_{L^2(D)} \leq k\, \|w \|_V, \;
 \|w \|_{L^2(\partial D)} \leq \frac{1}{k \sqrt{L}} \| w \|_V,
\end{equation*}
and hence using the definition of $\La$ in \eqref{eq:pdeoper_omega} and
$\M_2$ in \eqref{MU} we obtain
\begin{equation*}
 \|\La w \|_{L^2(D)} \,\leq\, k\,(1+n_{\max}) \|w \|_{V}\,, \qquad
 \|\M_2  w \|_{L^2(D)} \,\leq\,
 \Big(L+L\,|\widehat \beta_2|+\frac{|\alpha_2|}{k}\Big) \|w \|_{V}.
\end{equation*}
In addition, for all $j \geq 1$, we have
\begin{align*} 
  \|\nabla \cdot (\bsx\,\psi_j)\|_{L^\infty(D)}
  &\,=\, \| (\nabla \cdot \bsx)\psi_j + \bsx \cdot \nabla\psi_j\|_{L^\infty(D)} \nonumber\\
  &\,\le\, d\,\|\psi_j\|_{L^\infty(D)} + L\,\|\nabla\psi_j\|_{L^\infty(D)}
  \,\leq\, (d+1)\, \|\psi_j\|_{W^{1,\infty}(D)}.
\end{align*}
For $z,w \in V$, using the above bounds in~\eqref{eq:Rj}, the definition
of $C_R(kL)$ in~\eqref{eq:CR}, applying the triangle and Cauchy-Schwarz
inequalities, we obtain for $j \ge 1$,
\begin{align} \label{eq:Rj-bd}
 &|R_j(z,w)| \nonumber\\
 &\,\le\, A\,\|\psi_j\|_{L^\infty(D)}\,\|z\|_{L^2(D)}\|\La w\|_{L^2(D)} \nonumber \\
 &\qquad + k^2\,\|\psi_j\|_{L^\infty(D)}\,\|\M_2 z\|_{L^2(D)}\,\|w\|_{L^2(D)}
         + A\,\|\psi_j\|_{L^\infty(D)}\,\|\La z\|_{L^2(D)}\,\|w\|_{L^2(D)} \nonumber \\
 &\qquad + \big|-\alpha_1-\alpha_2-\mathrm{i}\,k\,L\,(\widehat\beta_1-\widehat\beta_2)\big|
 \,k^2\,\|\psi_j\|_{L^\infty(D)}\,\|z\|_{L^2(D)}\,\|w\|_{L^2(D)} \nonumber \\
 &\qquad + k^2\,\|\nabla\cdot(\bsx \psi_j) \|_{L^\infty(D)}\,\|z\|_{L^2(D)}\,\|w\|_{L^2(D)} \nonumber \\
 &\qquad + k^2\,\|\bsx\cdot \bfn\|_{L^\infty(D)}\,\|\psi_j\|_{L^\infty(D)}\,\|z\|_{L^2(\partial D)}\,\|w\|_{L^2(\partial D)}
 \nonumber \\
 &\,\le\, A\,(1+n_{\max})\,\|\psi_j\|_{L^\infty(D)}\,\|z\|_V\|w\|_V \nonumber \\
 &\qquad + k\,\Big(L+L\,|\widehat \beta_2|+\frac{|\alpha_2|}{k}\Big)\,\|\psi_j\|_{L^\infty(D)}\,\|z\|_V\,\|w\|_V
         + A\,(1+n_{\max})\,\|\psi_j\|_{L^\infty(D)}\,\|z\|_V\,\|w\|_V \nonumber \\
 &\qquad + \big|-\alpha_1-\alpha_2-\mathrm{i}\,k\,L\,(\widehat\beta_1-\widehat\beta_2)\big|
 \,\|\psi_j\|_{L^\infty(D)}\,\|z\|_V\,\|w\|_V \nonumber \\
 &\qquad + (d+1)\,\|\psi_j\|_{W^{1,\infty}}\,\|z\|_V\,\|w\|_V
 + \hat\mu\,\|\psi_j\|_{L^\infty(D)}\,\|z\|_V\,\|w\|_V
 \nonumber \\
 &\,\le\, \Big[2\,A\,(1+n_{\max}) + kL\,\big(1+|\widehat \beta_2|)+ |\alpha_2|
  + \big|-\alpha_1-\alpha_2-\mathrm{i}\,k\,L\,(\widehat\beta_1-\widehat\beta_2)\big| + d+1 + \hat\mu \Big] \nonumber\\
  &\qquad \times \|\psi_j\|_{W^{1,\infty}}\,\|z\|_V\,\|w\|_V \nonumber \\
  &\,=\, C_R(kL)\, \|\psi_j\|_{W^{1,\infty}(D)}\, \|z\|_{V}\, \|w\|_{V}.
\end{align}

Similarly, with $u(.,\bsy) \in V$ being the solution of~\eqref{eq:ses_var}
and applying \eqref{eq:G-bd} and the above bounds in~\eqref{eq:Snu} and
\eqref{eq:Fnu}, for any $\bsnu\in\indx$, including $\bsnu=\bszero$, we
obtain
\begin{align} \label{eq:S-bd}
 &|S_{\bsnu}(u,w) + T_\bsnu(w)| \nonumber\\
 &\,\le\,
 \begin{cases}
  \|\antilin_\bsy\|_{V^*}\,\|w\|_V
  & \mbox{if } \bsnu = \bszero, \\
  A\,\|\psi_j\|_{L^\infty(D)} \|w\|_{L^2(D)}\, \|f\|_{L^2(D)}
  & \mbox{if } \bsnu = \bse_j, \\
 A k^2 \!\!\displaystyle\sum_{j\in\supp(\bsnu)} \sum_{\ell\in\supp(\bsnu-\bse_j)} \!\!\!\!\!
 \nu_j (\bsnu-\bse_j)_\ell\, \|\psi_j\|_{L^\infty(D)}\, \|\psi_\ell\|_{L^\infty(D)}\,
 \|\partial^{\bsnu-\bse_j-\bse_\ell}u\|_{L^2(D)} &\!\!\!\!\!\|w\|_{L^2(D)} \\
  & \mbox{otherwise},
  \end{cases}
 \nonumber\\
 &\,\le\,  \bbS_{\bsnu}(u)\, \|w\|_{V}\,,
\end{align}
where
\begin{align} \label{eq:bSnu}
 &\bbS_{\bsnu}(u) \nonumber \\
 &\,:=\,
 \begin{cases}
  C_{\rm func}(kL)\,\big(L\, \|f\|_{L^2(D)}+L^{1/2} \|g\|_{L^2(\partial D)} \big) & \mbox {if} \quad \bsnu = \bszero, \\
  \displaystyle \frac{A}{k}\, \|\psi_j\|_{L^\infty(D)}  \|f\|_{L^2(D)} & \mbox{if} \quad \bsnu = \bse_j, \\
  \displaystyle A \sum_{j\in\supp(\bsnu)} \sum_{\ell\in\supp(\bsnu-\bse_j)}
  \nu_j (\bsnu-\bse_j)_\ell \, \|\psi_j\|_{L^\infty(D)} \|\psi_\ell\|_{L^\infty(D)} \,\|\partial^{\bsnu-\bse_j-\bse_\ell} u\|_{V} \,
  & \mbox{otherwise}.
  \end{cases}
\end{align}

Taking now $w = \partial^\bsnu u(\cdot,\bsy)$ in~\eqref{eq:der_rec}, using
the coercivity property~\eqref{eq:coer_res} as lower bound, and using
\eqref{eq:Rj-bd} and \eqref{eq:S-bd} as upper bounds, we obtain
\[
  C_{\rm coer}
  \|\partial^{\bsnu}u\|_{V}^2 \,\leq \left |\bilin_\bsy(\partial^\bsnu u,\partial^\bsnu u)\right |
  \,\leq\, \bigg( \sum_{j\in\supp(\bsnu)} \!\!\!\!\!\nu_j\, C_R(kL)\, \|\psi_j\|_{W^{1,\infty}(D)}\, \|\partial^{\bsnu-\bse_j} u\|_V
  + \bbS_{\bsnu}(u) \bigg)\, \| \partial^{\bsnu}u\|_V\,,
\]
and hence (now showing dependence on $\bsy$)
\begin{align}\label{eq:der_bd1}
 \|\partial^{\bsnu}u(\cdot,\bsy)\|_{V}
   \,\le\, \frac{1}{C_{\rm coer}} \bigg(
   C_R(kL) \sum_{j\in\supp(\bsnu)} \nu_j\,\|\psi_j\|_{W^{1,\infty}(D)}\,
   \|\partial^{\bsnu-\bse_j}u(\cdot,\bsy)\|_{V}  + \bbS_\bsnu(u(\cdot,\bsy))\bigg)\,.
\end{align}

The desired result now follows from \eqref{eq:bSnu} and \eqref{eq:der_bd1}
by applying Lemma~\ref{lem:recur4} with
\begin{align*}
 &\bbA_\bsnu \,=\, \| \partial^\bsnu u(\cdot,\bsy)\|_V \,,\quad
 B \,=\, \frac{C_{\rm func}(kL)}{C_{\rm coer}} \big(L\, \|f\|_{L^2(D)}+L^{1/2} \|g\|_{L^2(\partial D)} \big)\,,
 \\
 &\Psi_j \,=\, \|\psi_j\|_{W^{1,\infty}(D)}\,,\quad
 c_0 \,=\, \frac{C_R(kL)}{C_{\rm coer}} + \frac{A}{kL\,C_{\rm func}(kL)}\,, \quad
 c_1 \,=\, \frac{C_R(kL)}{C_{\rm coer}}\,, \quad
 c_2 \,=\, \frac{A}{C_{\rm coer}}\,.
\end{align*}
The value of $B$ is determined by taking $\bsnu = \bszero$ in
\eqref{eq:bSnu} and \eqref{eq:der_bd1}. The values of $c_1$ and $c_2$
follow easily by taking $|\bsnu|\ge 2$ in \eqref{eq:bSnu} and
\eqref{eq:der_bd1}. The remaining case of $|\bsnu| = 1$ is slightly more
complicated: taking $\bsnu = \bse_j$ in \eqref{eq:bSnu} and
\eqref{eq:der_bd1} yields
\begin{align*}
 \|\partial^{\bse_j}u(\cdot,\bsy)\|_{V}
 &\,\le\, \frac{1}{C_{\rm coer}} \bigg(
   C_R(kL)\,\|\psi_j\|_{W^{1,\infty}(D)}\,\|u(\cdot,\bsy)\|_V
   + \frac{A}{k}\,\|\psi_j\|_{L^\infty(D)}\,\|f\|_{L^2(D)}\bigg) \\
 &\,\le\, \Big(\frac{C_R(kL)}{C_{\rm coer}} + \frac{A}{kL\,C_{\rm func}(kL)}\Big)\,\Psi_j\, B\,,
\end{align*}
which gives the value of $c_0$. With these values we obtain $C_{\rm
regu}(kL) = \max\left\{c_0, 2c_1,\sqrt{2c_2}\right\}$ as given in
\eqref{eq:c-regu}. This completes the proof.
\end{proof}

\section{Stochastic refractive index dimension truncation}\label{sec:dim_trunc}

For simulation of the stochastic wave propagation induced by the
refractive index, we need to truncate the infinitely many terms in the
ansatz~\eqref{eq:axy-unif}. To analyze the dimension truncation error, it
is convenient to introduce an operator theoretical framework which
incorporates the boundary condition.

Recalling \eqref{eq:pde_y}, for each $\bsy \in U$ we now define the
operator $\Lad(\bsy): V \rightarrow L^2(D) \times L^2(\partial D)$ by
\begin{equation}\label{eq:pde_y_op}
 [\Lad(\bsy) w](\bsx, \bsxt)
 \,:=\,
 \begin{pmatrix}
 \left[\Delta +k^2\,n(\bsx, \bsy)\right] w(\bsx)  \\
 \displaystyle\frac{\partial w}{\partial \bfn}(\bsxt) - \ii \,k\,w(\bsxt)
 \end{pmatrix},
 \qquad \bsx \in  D, \quad \widetilde{\bsx} \in \partial D.
\end{equation}
Then \eqref{eq:pde_y} can be expressed as
\begin{equation*} 
 [\Lad(\bsy) u(\cdot,\bsy)](\bsx, \bsxt)
 \,=\,
 \begin{pmatrix} -f(\bsx) \\ g(\bsxt) \end{pmatrix},
 \qquad \bsx \in  D, \quad \widetilde{\bsx} \in \partial D.
\end{equation*}
We equip $L^2(D) \times L^2(\partial D)$ with the weighted product space
norm
\begin{equation*} 
 \big\|\sbinom{f}{g}\big\|_{L^2(D)\times L^2(\partial D)}
\,:=\, L\, \| f \|_{L^2(D)} + \sqrt{L} \,\| g \|_{L^2(\partial D)},
\qquad f\in L^2(D), \quad g\in L^2(\partial D).
\end{equation*}
It is easy to check that $\Lad(\bsy)$ is a bounded linear operator.

From Theorem~\ref{COERTHM} we conclude that $\Lad(\bsy)$ is boundedly
invertible for all $\bsy\in U$. Indeed, for any $\sbinom{f}{g} \in L^2(D)
\times L^2(\partial D)$ we can write
\begin{equation*} 
 u(\cdot, \bsy) \,=\, \Ladinvy \sbinom{-f}{g},
 \quad\mbox{with}\quad
 \big \| \Ladinvy \sbinom{-f}{g} \big\|_V
 \,\leq\, \frac{C_{\rm func}(kL)}{C_{\rm coer}}\, \big\| \sbinom{f}{g} \big\|_{L^2(D)\times L^2(\partial D)},
\end{equation*}
and therefore
\begin{equation}\label{eq:inv-Lu-norm}
 \big\| \Ladinvy \big\| \,\le\, \frac{ C_{\rm func}(kL)}{C_{\rm coer}},
\end{equation}
which is bounded independently of the wavenumber if $kL \ge 1$.

Corresponding to~\eqref{eq:axy-unif}, for a truncation parameter $s$ we
consider a truncated refractive index (essentially by setting $y_j=0$ for
$j> s$)
\begin{equation*} 
 n_s(\bsx,\bsy)
 \,=\, n_s(\bsx,\bsy_{\{1:s\}}) \,=\, n_0(\bsx) + \sum_{j= 1}^s y_j\, \psi_j(\bsx),
\end{equation*}
and define the operator $\Lads(\bsy) = \Ladys: V \rightarrow L^2(D) \times
L^2(\partial D)$ as in~\eqref{eq:pde_y_op} but with $n$ replaced by $n_s$.
Then we have also
\begin{equation}\label{eq:inv-Lu-norm-s}
 u_s(\cdot, \bsy) = [\Lad_s(\bsy)]^{-1} \sbinom{-f}{g}
 \qquad  \text{and} \qquad \big\| [\Lad_s(\bsy)]^{-1} \big\| \,\le\, \frac{ C_{\rm func}(kL)}{C_{\rm coer}}.
\end{equation}

For a fixed truncated dimension $s$, in the following theorem we will
estimate  the approximation error $(I-I_s)(G(u))$, where the
infinite-dimensional integral $I$ and the finite-dimensional integral
$I_s$ are as defined in~\eqref{eq:int-unif}--\eqref{eq:int-unifs}. The
proof of the estimate is based on the dimension truncation error of the
integrand
\[
  u(\cdot, \bsy) - u_s(\cdot, \bsy)
  \,=\, \big( [\Lad(\bsy)]^{-1} - [\Lad_s(\bsy)]^{-1} \big) \sbinom{-f}{g}
\]
by a Neumann series argument. The first critical step is to recogonize
that we can write the difference operator $[\Lad(\bsy) - \Lad_s(\bsy)]: V
\rightarrow L^2(D) \times L^2(\partial D)$ as
\begin{equation}\label{eq:L-Ls}
 [\Lad(\bsy) - \Lad_s(\bsy)] w \,=\, \sum_{j \geq s+1} y_j\, \Laj\, w,
\end{equation}
with operators $\Laj : V \rightarrow L^2(D) \times L^2(\partial D)$
defined as
\begin{equation}\label{eq:Tj}
  \Laj\, w \,:=\, k^2\, \binom{\psi_j\,w}{0}, \qquad j \geq 1.
 \end{equation}
The proof follows the general argument of \cite{Gan18} but there are some
key differences which mean that we do not need to impose the kind of small
perturbation assumption discussed in Appendix~\ref{append:pert}.

For developing the dimension truncation and QMC-FEM analysis in this
article, we will impose the following assumptions on the perturbation
functions $\psi_j$ in~\eqref{eq:axy-unif}:
\begin{enumerate}
\item[(A3)]\label{A3}%
 The sequence $\psi_j$ is ordered:  $\|\psi_1\|_{L^\infty(D)}\ge
\|\psi_2\|_{L^\infty(D)} \ge\cdots$.
\item[(A4)]\label{A4}%
There exists $p_0\in (0,1)$ and $K_0\in\bbR$ independently of $k$ such
that
\begin{align} \label{eq:K0}
  \sum_{j\ge 1} \Big[(kL + 1)\,\|\psi_j\|_{L^\infty(D)}\Big]^{p_0} \,\le\, K_0 \,<\, \infty.
\end{align}
\item[(A5)]\label{A5}%
There exists $p_1 \in (0,1)$ and $K_1\in\bbR$ independently of $k$
such that
\begin{align} \label{eq:K1}
  \sum_{j\ge 1} \Big[\big(kL + (kL)^{-1}\big)\, \|\psi_j\|_{W^{1,\infty}(D)} \Big]^{p_1} \,\le\, \,K_1\, < \infty.
\end{align}
\end{enumerate}

These conditions are similar to counterpart conditions assumed for the
diffusion model in~\cite{KSS12} and also for general class of operator
equations in \cite{DKLNS14,DKLS16}, but now with explicit dependence
on~$kL$. We use the assumption (A5) in the next section to obtain QMC
error bounds.

\begin{theorem} \label{theorem:truncation}
Let the assumptions \textnormal{(A0)}--\textnormal{(A4)} and parameter
restrictions in Theorem~\ref{COERTHM} hold. For every $\bsy \in U$, $f\in
L^2(D)$ and $g\in L^2(\partial D)$, let $u(\cdot,\bsy) \in V$ be the
unique solution of~\eqref{eq:pde_y}, and for each $s\in\bbN$ let
$u_s(\cdot,\bsy)$ denote the solution of the truncated version
of~\eqref{eq:pde_y} with $n$ replaced by~$n_s$. Then for every linear
functional $G\in V^*$, there exist a constant $C$ independent of $s, f, g,
G$ and $kL$ such that
\begin{align}   \label{eq:trunc-unif}
  |(I-I_s)(G(u))| &\,=\, |I(G(u-u^s))|
  \,\leq\, C\,\frac{C_{\rm func}(kL)}{C_{\rm coer}}\,
  \|\sbinom{f}{g}\|_{L^2(D) \times L^2(\partial D)}\,\|G\|_{V^*}\,
  s^{-\frac{2}{p_0}+1},
\end{align}
which is bounded independently of the wavenumber if $kL \geq 1$.
\end{theorem}

\begin{proof}
In this proof we will suppress the dependence on $\bsy$ to simplify our
notation where possible. We will begin by expanding $u - u_s = (\Lad^{-1}
- \Lad_s^{-1}) \sbinom{-f}{g}$ in a Neumann series for sufficient
large~$s$.  Writing $\Lad^{-1} = (\calI + \Lad_s^{-1} (\Lad -
\Lad_s))^{-1} \Lad_s^{-1}$, we need to first ensure that $\|-\Lad_s^{-1}
(\Lad - \Lad_s)\| < 1$.

For each $j\ge 1$ and $w\in V$, we have from \eqref{eq:Tj} that
\begin{align*}
 \big\|\Lad_s^{-1}\, \Laj\, w \big\|_V
  & \,=\, \left \|k^2\, \Lad_s^{-1} \sbinom{\psi_j\, w}{0} \right \|_V \\
  & \,\le\, k^2 \, \frac{C_{\rm func}(kL)}{C_{\rm coer}}\,
  \big\|\sbinom{\psi_j\,w}{0} \big\|_{L^2(D)\times L^2(\partial D)}
  \,\le\, k L\,\frac{C_{\rm func}(kL)}{C_{\rm coer}}\, \|\psi_j\|_{L^\infty(D)} \, \|w\|_V,
\end{align*}
where we used $ \| \sbinom{w}{0} \|_{L^2(D)\times L^2(\partial D)} = L\,
\| w \|_{L^2(D)} \le \frac{L}{k}\|w\|_V$. Thus $\Lad_s^{-1}\, \Laj$ is a
bounded operator from $V$ to $V$, with norm
\begin{align} \label{eq:bj}
  \big\| \Lad_s^{-1}\, \Laj \big\|
  \,\le\, kL\,\frac{C_{\rm func}(kL)}{C_{\rm coer}}\, \|\psi_j\|_{L^\infty(D)}
  \,=:\, b_j.
\end{align}
Hence from~\eqref{eq:L-Ls} we have for all $\bsy\in U$,
\begin{align*}
  \big\| - \Lad_s^{-1}(\Lad-\Lad_s) \big\|
  \,=\,  \bigg\|\sum_{j\ge s+1} y_j\, \Lad_s^{-1}\, \Laj \bigg\|
  \,\le\, \frac{1}{2} \sum_{j\ge s+1} b_j.
\end{align*}
Since $kL\,C_{\rm func}(kL) = \calO(kL+1)$, from Assumptions~(A3) and~(A4)
we know that the sequence $\{b_j\}_{j\ge 1}$ is nonincreasing, and that
\begin{align} \label{eq:sum-p0}
  \sum_{j\ge 1} b_j^{p_0} \,\le\, r_0\,K_0 \,<\, \infty,
\end{align}
for some constant $r_0$ independent of the wavenumber $k$.

Let $s^*$ be such that $\sum_{j \geq s^*+1} b_j \leq \frac{1}{2}$,
implying that $\left   \|-\Lad_s^{-1}(\Lad-\Lad_s)\right \| \leq
\frac{1}{4}$. Then for all $s \geq s^*$, by the bounded invertibility of
$\Lady$ and $\Ladys$ for all $\bsy \in U$, we can write the inverse of
$\Lad$ in terms of the Neumann series, as
\[
  \Lad^{-1} \,= \, \big(\calI + \Lad_s^{-1} (\Lad - \Lad_s)\big)^{-1} \Lad_s^{-1}
  \,=\, \sum_{\ell\ge 0} \big(-\Lad_s^{-1} (\Lad-\Lad_s)\big)^\ell \Lad_s^{-1}\,.
\]
Then, using representations~\eqref{eq:inv-Lu-norm}, \eqref{eq:inv-Lu-norm-s} and~\eqref{eq:Tj}, we obtain
\begin{align*}
  u - u_s
  \,=\, \left (\Lad^{-1}-\Lad_s^{-1} \right) \sbinom{-f}{g}
  &\,=\, \sum_{\ell\ge 1} \big(-\Lad_s^{-1} (\Lad-\Lad_s)\big)^\ell \Lad_s^{-1} \sbinom{-f}{g} \\
  &\,=\, \sum_{\ell\ge 1} (-1)^\ell \bigg(\sum_{j\ge s+1} y_j\, \Lad_s^{-1}\, \Laj\bigg)^\ell u_s  \\
  &\,=\, \sum_{\ell\ge 1} (-1)^\ell \sum_{\bseta\in\{s+1:\infty\}^\ell}
  \prod_{i=1}^\ell \Big(y_{\eta_i}\, \Lad_s^{-1}\, \calT_{\eta_i}\Big) u_s,
\end{align*}
where we use the shorthand notation $\{s+1:\infty\}^\ell =
\{s+1,s+2,\ldots,\infty\}^\ell$.

Thus we can write
\begin{align*}
  &\int_U G(u - u_s) \,\rd\bsy
 \,=\, \sum_{\ell\ge 1} (-1)^\ell
 \sum_{\bseta\in\{s+1:\infty\}^\ell} \int_U
 G \bigg[\bigg(\prod_{i=1}^\ell (y_{\eta_i}\,  \Lad_s^{-1}\, \calT_{\eta_i})\bigg)
  u_s\bigg]\,\rd\bsy \\
  &\,=\, \sum_{\ell\ge 1} (-1)^\ell
  \sum_{\bseta\in\{s+1:\infty\}^\ell}
  \bigg(\int_{U_{s+}} \prod_{i=1}^\ell y_{\eta_i}\,\rd\bsy_{\{s+1:\infty\}}\bigg)
  \bigg(\int_{U_s} G \bigg[\bigg(\prod_{i=1}^\ell ( \Lad_s^{-1}\, \calT_{\eta_i})\bigg)
  u_s\bigg]\,\rd\bsy_{\{1:s\}} \bigg),
\end{align*}
where we separated the integrals for $\bsy_{\{1:s\}} \in U_s :=
\left[-\frac{1}{2},\frac{1}{2}\right]^s$ and $\bsy_{\{s+1:\infty\}} :=
(y_j)_{j\ge s+1} \in U_{s+} := \{(y_j)_{j\geq s+1}\,:\,y_j \in
\left[-\frac{1}{2},\frac{1}{2}\right],\, j\geq s+1\}$, which is an
essential step of this proof. The integral over $\bsy_{\{s+1:\infty\}}$ is
nonnegative due to the simple yet crucial observation that
\begin{align} \label{eq:simple}
\int_{-\frac{1}{2}}^{\frac{1}{2}} y_j^{n}\,\rd y_j \,=\,
  \begin{cases}
  0  & \mbox{if $n$ is odd}, \\
  \frac{1}{2^{n}(n+1)}  & \mbox{if $n$ is even}.
  \end{cases}
\end{align}
The integral over $\bsy_{\{1:s\}}$ can be estimated, using  \eqref{eq:inv-Lu-norm-s} and \eqref{eq:bj}, as
\begin{align*}
  \bigg|\int_{U_s} G \bigg[\bigg(\prod_{i=1}^\ell ( \Lad_s^{-1}\, \calT_{\eta_i})\bigg)
  u_s\bigg]\,\rd\bsy_{\{1:s\}} \bigg|
  &\,\le\, \|G\|_{V^*} \sup_{\bsy_{\{1:s\}}\in U_s} \bigg\|\bigg(\prod_{i=1}^\ell ( \Lad_s^{-1}\, \calT_{\eta_i})\bigg)\bigg\| \|u_s\|_V \\
  &\,\le\, \frac{C_{\rm func}(kL)}{C_{\rm coer}}\,\|\sbinom{f}{g}\|_{L^2(D) \times L^2(\partial D)}\,\|G\|_{V^*}\, \prod_{i=1}^\ell b_{\eta_i}\,.
\end{align*}
Hence, with the abbreviation
\begin{align} \label{eq:C1-trunc}
  C_1 \,:=\, \frac{C_{\rm func}(kL)}{C_{\rm coer}}\,\|\sbinom{f}{g}\|_{L^2(D) \times L^2(\partial D)}\,\|G\|_{V^*},
\end{align}
we obtain
\begin{align*}
  &\bigg|\int_U G(u - u_s) \,\rd\pmb y\bigg|
  \,\le\, C_1 \,
  \sum_{\ell\ge 1}
  \sum_{\bseta\in\{s+1:\infty\}^\ell}
  \bigg(\int_{U_{s+}} \prod_{i=1}^\ell y_{\eta_i}\,\rd\bsy_{\{s+1:\infty\}}\bigg)\,
  \prod_{i=1}^\ell b_{\eta_i} \\
  &\,=\, C_1\,
  \sum_{\ell\ge 1}
\int_{U_{s+}}
  \sum_{\bseta\in\{s+1:\infty\}^\ell}
  \bigg(\prod_{i = 1}^\ell y_{\eta_i} b_{\eta_i}\bigg)\,
  \rd\bsy_{\{s+1:\infty\}}
  \,=\, C_1\,
  \sum_{\ell\ge 1} \int_{U_{s+}}
  \bigg(\sum_{j\ge s+1} y_j\,b_j\bigg)^\ell\,\rd\bsy_{\{s+1:\infty\}}.
\end{align*}
Using the multinomial theorem with multi-index $\bsnu$ and
$\binom{\ell}{\bsnu} = \ell!/(\prod_{j\ge 1}\nu_j!)$, we can write
\begin{align*}
  &\bigg|\int_U G(u - u_s) \,\rd\pmb y\bigg|
  \,\le\, C_1\,
  \sum_{\ell\ge 1} \int_{U_{s+}}
  \sum_{\satop{|\bsnu|=\ell}{\nu_j=0\;\forall j\le s}} \binom{\ell}{\bsnu}
  \prod_{j\ge s+1} (y_j\,b_j)^{\nu_j} \,\rd\bsy_{\{s+1:\infty\}} \\
  &\,=\, C_1\,
  \sum_{\ell\ge 1}
  \sum_{\substack{|\bsnu|=\ell \\ \nu_j=0\;\forall j\le s}} \binom{\ell}{\bsnu}
  \bigg(\prod_{j\ge s+1}\int_{-\frac{1}{2}}^{\frac{1}{2}} y_j^{\nu_j}\,\rd y_j\bigg)\,
  \prod_{j\ge s+1} b_j^{\nu_j} \\
  &\,\leq\, C_1\,
  \sum_{\ell\ge 2\;{\rm even}}
  \sum_{\substack{|\bsnu|=\ell \\ \nu_j=0\;\forall j\le s \\ \nu_j\;{\rm even}\;\forall j\ge s+1}} \binom{\ell}{\bsnu}
  \prod_{j\ge s+1} b_j^{\nu_j}
  \,=\, C_1\,
  \sum_{\ell'\ge 1}
  \sum_{\substack{|\bsnu|=2\ell' \\ \nu_j=0\;\forall j\le s \\ \nu_j\;{\rm even}\;\forall j\ge s+1}} \binom{2\ell'}{\bsnu}
  \prod_{j\ge s+1} b_j^{\nu_j},
\end{align*}
where the last inequality follows from~\eqref{eq:simple}.

Now we split the sum into a sum over $\ell'\ge \ell^*$ (dropping the
condition ``$\nu_j$ even'') and the initial terms $1\le \ell' < \ell^*$
(substituting $\nu_j = 2\nu_j'$) to obtain the estimate
\begin{align} \label{eq:bound2}
  &\bigg|\int_U G(u - u_s) \,\rd\pmb y\bigg|
  \,\le\, C_1\,
  \sum_{\ell'\ge \ell^*} \!\!
  \sum_{\satop{|\bsnu|=2\ell'}{\satop{\nu_j=0\;\forall j\le s}{\nu_j\;{\rm even}\;\forall j\ge s+1}}}
  \!\!\!\!\!\binom{2\ell'}{\bsnu}
  \prod_{j\ge s+1} b_j^{\nu_j}
  \,+\, C_1\,
  \sum_{1\le \ell' < \ell^*}
  \sum_{\satop{|\bsnu'|=\ell'}{\nu_j'=0\;\forall j\le s}} \binom{2\ell'}{2\bsnu'}
  \prod_{j\ge s+1} (b_j^2)^{\nu_j'} \nonumber\\
  &\,\le\, C_1\,
  \sum_{\ell'\ge \ell^*}
  \sum_{\satop{|\bsnu|=2\ell'}{\nu_j=0\;\forall j\le s}} \binom{2\ell'}{\bsnu}
  \prod_{j\ge s+1} b_j^{\nu_j}
  \;+\; C_1\,
  \sum_{1\le \ell' < \ell^*} \frac{(2\ell')!}{(\ell')!}
  \sum_{\satop{|\bsnu'|=\ell'}{\nu_j'=0\;\forall j\le s}} \binom{\ell'}{\bsnu'}
  \prod_{j\ge s+1} (b_j^2)^{\nu_j'} \nonumber\\
  &\,\le\,
  C_1\, \sum_{\ell'\ge \ell^*} \bigg(\sum_{j\ge s+1} b_j\bigg)^{2\ell'}
  \;+\; C_1\, \sum_{1\le \ell' < \ell^*} \frac{(2\ell')!}{(\ell')!} \bigg(\sum_{j\ge s+1} b_j^2\bigg)^{\ell'} \nonumber\\
  &\,\le\, C_1\, \frac{(\sum_{j\ge s+1} b_j)^{2\ell^*}}{1-(\sum_{j\ge s+1} b_j)^2}
  \;+\; C_1\, \frac{(2\ell^*-2)!}{(\ell^*-1)!}\, \frac{\sum_{j\ge s+1} b_j^2}{1-(\sum_{j\ge s+1} b_j^2)},
\end{align}
where we used the multinomial theorem and the geometric series formula,
noting that for $s\ge s^*$ we have $\sum_{j \geq s+1} b_j^2 \leq
\sum_{j\geq s+1} b_j \leq \frac{1}{2}$.

From \cite[Theorem 5.1]{KSS12} we know that
\begin{align} \label{eq:sum1}
 \sum_{j \geq s+1} b_j
 \leq \min{\bigg( \frac{1}{\frac{1}{p_0}-1},1 \bigg)} \bigg(\sum_{j \geq 1} b_j^{p_0} \bigg)^{\frac{1}{p_0}}
 s^{- \frac{1}{p_0}+1}\,.
\end{align}
With a similar argument we can show that
\begin{align} \label{eq:sum2}
 \sum_{j \geq s+1} b_j^2
 \leq \frac{1}{\frac{2}{p_0}-1} \bigg(\sum_{j \geq 1} b_j^{p_0} \bigg)^{\frac{2}{p_0}}
 s^{- \frac{2}{p_0}+1}\,.
\end{align}
Using the estimates \eqref{eq:sum1} and \eqref{eq:sum2} for the numerators
in \eqref{eq:bound2} and bounding the sums in the denominators by $1/2$,
we see that the first term in \eqref{eq:bound2} is $\calO(s^{-2\ell^*
(1/p_0-1)})$ while the second term is $\calO(s^{-(2/p_0-1)})$. We
therefore choose $\ell^*$ such that $2\ell^*(1/p_0-1) \ge 2/p_0-1$, i.e.,
$\ell^* := \lceil (2-p_0)/(2-2p_0)\rceil$. Hence, for all $s\ge s^*$ we
arrive at
\begin{align} \label{eq:bound3}
  \bigg|\int_U G(u - u_s) \,\rd\pmb y\bigg|
  &\,\le\, C_1\,C_2\, s^{-\frac{2}{p_0}+1},
\end{align}
where $C_2$ is a constant depending on $p_0$ and $K_0$, and is independent
of $k$.

It remains to derive the bound for $s< s^*$. Using~\eqref{eq:inv-Lu-norm}
and \eqref{eq:inv-Lu-norm-s} we have the estimate
\begin{align} \label{eq:bound4}
  \bigg|\int_U G(u - u_s) \,\rd\bsy\bigg|
  &\,\le\, \|G\|_{V^*} \sup_{\bsy\in U} \Big(\|u(\cdot,\bsy)\|_V + \|u_s(\cdot,\bsy)\|_V\Big) \nonumber\\
  &\,\le\, \frac{2\,C_{\rm func}(kL)}{C_{\rm coer}}\, \|\sbinom{f}{g}\|_{L^2(D) \times L^2(\partial D)}\,\|G\|_{V^*}
  \,\le\, 2\,C_1\cdot (s^*)^{\frac{2}{p_0}-1} s^{-\frac{2}{p_0}+1},
\end{align}
where we used $s^*/s>1$ and the definition of $C_1$ in
\eqref{eq:C1-trunc}. We now use \eqref{eq:sum-p0} to get an upper bound on
\eqref{eq:sum1} involving $K_0$, and choose $s^*$ such that when $s$
replaced by $s^*$ this upper bound is at most $1/2$. Consequently $s^*$ is
a constant depending on $p_0$ and $K_0$, and is independent of $k$.

Combining now \eqref{eq:bound3} and \eqref{eq:bound4}, and plugging in the
definition \eqref{eq:C1-trunc} for $C_1$, we obtain the required result
for all values of $s$.
\end{proof}

\section{Finite element discretizations}\label{sec:fem}

In this section first we consider a high-order FEM for computationally
solving   the sign-definite sesquilinear formulation. For each $\bsy \in
U$, having quantified the error resulting from  dimension truncation  of
the stochastic refractive index field $n(\cdot, \bsy)$ by
$n_s(\cdot,\bsy_{\left\{1:s\right\}})$ to approximate the solution
$u(\cdot, \bsy)$ of~\eqref{eq:ses_var} by the solution $u_s(\cdot,
\bsy_{\left\{1:s\right\}})$ satisfying
\begin{equation}\label{eq:trun_ses_var}
 \bilin_{\bsy_{\{1:s\}}}(u_s,v) = \antilin_{\bsy_{\{1:s\}}}(v), \quad \text{for all }\,  v\in V,
\end{equation}
we consider the spatial Galerkin FEM approximation of $u_s$ by $u_{s,h}$.
To this end, we choose a finite dimensional  subspace   $V_h^p \subset
H^2(\Omega)$ spanned by splines of degree $p \geq 2$ on a tessellation (of
at least $C^1$-elements with maximum width $h$) of $D$. The space $V_h^p$
is chosen so that the following approximation property holds: for $0 \leq
t \leq  2$ and for any $v \in H^{t^*}(\Omega)$  with $t^* \geq t+1$,
\begin{equation} \label{eq:approx_prop}
\inf\limits_{w_h\in V_h^p} \|v-w_h\|_{H^{t}} \,\le\, C_{\rm appr}\, h^{\min\{p+1, t^*\}-t},
\end{equation}
and the constant $C_{\rm appr}$ depends on the chosen norm of $v$.

For each $\bsy \in U$, the FEM approximation $u_{s,h}(\cdot,
\bsy_{\left\{1:s\right\}})\in V_h^p$ to the unique solution $u_s(\cdot,
\bsy_{\left\{1:s\right\}})$ of~\eqref{eq:trun_ses_var} is required to be
computed by solving the  linear algebraic system arising from the
finite-dimensional coercive variational form
\begin{equation}\label{eq:trun_fem_ses_var}
 \bilin_{\bsy_{\{1:s\}}}(u_{s,h},v) = \antilin_{\bsy_{\{1:s\}}}(v) \quad \text{for all }\,  v\in V_h^p.
\end{equation}
Since $V \subset H^{3/2}(\Omega)$, using~\eqref{eq:approx_prop}, the
coercivity and continuity of the sesquilinear form $
\bilin_{\bsy_{\{1:s\}}}$, Theorem~\ref{COERTHM} and Cea's Lemma, under
appropriate spatial regularity assumption of $u_s$
satisfying~\eqref{eq:trun_ses_var} and the degree $p \geq 2$ of the
splines, the high-order FEM approximation $u_{s,h}\in V_h^p$ satisfies the
following error bound:
\begin{equation}\label{eq:FEerror-unif}
  \|u_s(\cdot, \bsy_{\left\{1:s\right\}}) - u_{s,h}(\cdot, \bsy_{\left\{1:s\right\}})\|_V
  \,\le\, C_{\rm appr}\,
  \frac{C_{\rm cont}(kL)}{C_{\rm coer}}\,\|\sbinom{f}{g}\|_{L^2(D) \times L^2(\partial D)}\, h^{p-1}.
\end{equation}
Recall that $C_{\rm cont}(kL) = \calO\big(kL + (kL)^{-1}\big)$. This
highlights that the well known {\em pollution effect} is present in our
(and all known) FEM approximations (converging in $h$) for the Helmholtz
PDE in two and higher dimensions. While the pollution effect requires
large degrees of freedom (DoF) for large acoustic size $kL > 1$ using the
standard piecewise-linear ($p =1$) low-order FEM, we have demonstrated
in~\cite{GanMor17,GanMor20} that the pollution error can be efficiently
avoided by using high-order FEM ($p \geq 2$), even for solutions with
limited regularity.

In particular, as demonstrated in~\cite{GanMor20} using an efficient
construction of the space $V_h^p$, the number of DoF do not increase
substantially despite imposing higher continuity requirements needed for
larger degree splines. In~\cite{GanMor20}, for heterogeneous deterministic
models (that is, with $y_j = 0$ in~\eqref{eq:axy-unif} for all $j \geq 1$
and spatially dependent mean-field $n_0$)   and for various acoustic size
$kL$ values with sufficiently smooth solutions, we have numerically
demonstrated $p-1$ estimated  order of convergence (EOC) for $p = 2, 3, 4$
in the $V$ norm,  as stated in~\eqref{eq:FEerror-unif}, and also $p+1$ and
$p$ EOC, respectively, in the $H^0$-norm and the $H^1$-norm.

For the bounded linear functional $G \in V^*$, based on Nitsche arguments,
we obtain for all $\bsy\in U$
\begin{align} \label{eq:duality-unif}
 \left|G(u_s(\cdot, \bsy_{\left\{1:s\right\}})) - G(u_{s,h}(\cdot, \bsy_{\left\{1:s\right\}})) \right|
 \,\le\, C_{\rm appr}\,
  \frac{C_{\rm cont}(kL)}{C_{\rm coer}}\,
 \|\sbinom{f}{g}\|_{L^2(D) \times L^2(\partial D)}\,\|G\|_{V^*}\, h^{p},
\end{align}
and the same upper bound is obtained for for its integral counterpart
\[
  |I_s(G(u_s - u_{s,h}))|
  \,=\, |I(G(u_s(\cdot, \bsy_{\left\{1:s\right\}}))) - I(G(u_{s,h}(\cdot, \bsy_{\left\{1:s\right\}})))|.
\]
Thus the upper bound is of order $\calO\big((kL+(kL)^{-1})\, h^p\big) =
\calO\big((1+(kL)^{-1})\,(kL+1)\, h^p\big)$.

\section{Quasi-Monte Carlo integration}\label{sec:qmc}

For complete details of various QMC integration rules, we refer to the
survey \cite{DKS13} and extensive references therein; see also the survey
\cite{KN16} for some QMC theory applied in the context of PDE problems. In
the next two subsections we focus on two QMC rules.

\subsection{Randomly shifted lattice rules (first order convergence)}

For a fixed dimension truncation parameter $s$, we consider the integral
of a general complex-valued function $F$ defined over the $s$-dimensional
unit cube $[-\tfrac{1}{2},\tfrac{1}{2}]^s$
\begin{align*} 
  I_s(F)
  \,=\, \int_{[-\tfrac{1}{2},\tfrac{1}{2}]^s} F(\bsy)\, \rd\bsy\,,
\end{align*}
and we approximate this by a \emph{randomly shifted lattice rule}
\begin{align} \label{eq:qmc-rsh}
  Q_{s,N}(F;\bsDelta)
  \,=\, \frac{1}{N} \sum_{i=1}^N F\big(\{\bst_i + \bsDelta\} - \bshalf\big)\,,
\end{align}
where $\bst_1,\ldots,\bst_N \in [0,1]^s$ are deterministic lattice
cubature points, and $\bsDelta$ is a random shift which is drawn from the
uniform distribution on $[0,1]^s$. The braces in \eqref{eq:qmc-rsh}
indicate that we take the fractional part of each component in a vector,
while the subtraction of $\bshalf$ takes care of the translation from the
standard unit cube $[0,1]^s$ to $[-\frac{1}{2},\frac{1}{2}]^s$. The
lattice points are given by $\bst_i = \{\frac{i\bsz}{N}\}$ for
$i=1,\ldots,N$, where $\bsz\in\bbZ^s$ is known as the \emph{generating
vector} and it determines the quality of the lattice rule.

We apply the theory and construction of randomly shifted lattice rules in
\emph{weighted Sobolev spaces} to obtain first order convergence rates.
Loosely speaking, these spaces contain functions with square integrable
mixed first derivatives. The norm is given by
\begin{align*} 
  \|F\|_{s,\bsgamma}
  \,=\,
  \Bigg(
  \sum_{\setu\subseteq\{1:s\}}
  \frac{1}{\gamma_\setu}
  \int_{[-\tfrac{1}{2},\tfrac{1}{2}]^{|\setu|}}
  \bigg|\int_{[-\tfrac{1}{2},\tfrac{1}{2}]^{s-|\setu|}}
  \frac{\partial^{|\setu|}F}{\partial \bsy_\setu}(\bsy_\setu;\bsy_{\{1:s\}\setminus\setu})
  \,\rd\bsy_{\{1:s\}\setminus\setu}
  \bigg|^2
  \rd\bsy_\setu
  \Bigg)^{1/2},
\end{align*}
where $\{1:s\}$ is a shorthand notation for the set of indices
$\{1,2,\ldots,s\}$, $(\partial^{|\setu|}F)/(\partial \bsy_\setu)$ denotes
the mixed first derivative of $F$ with respect to the ``active'' variables
$\bsy_\setu = (y_j)_{j\in\setu}$, while $\bsy_{\{1:s\}\setminus\setu} =
(y_j)_{j\in\{1:s\}\setminus\setu}$ denotes the ``inactive'' variables. The
weights $\gamma_\setu$ moderate the relative importance between subsets of
variables. It is known that (see e.g., \cite[Theorem~5.1]{DKS13}), given
$N$ a prime power and the weights $\gamma_\setu$ as input, a generating
vector $\bsz$ can be obtained by the \emph{component-by-component} (CBC)
construction to achieve the root-mean-square error (with respect to the
random shift)
\begin{align} \label{eq:qmc1}
  \sqrt{\bbE_{\rm rqmc} \big[\big|I_s(F) - Q_{s,N}(F;\cdot)\big|^2\big]}
  \,\le\, \bigg(\frac{2}{N} \sum_{\emptyset\ne\setu\subseteq\{1:s\}} \gamma_\setu^\lambda\,
  [\varrho(\lambda)]^{|\setu|} \bigg)^{1/(2\lambda)}\,\|F\|_{s,\bsgamma}
  \quad\forall\;\lambda\in (\tfrac{1}{2},1],
\end{align}
where $
  \varrho(\lambda)
  \,=\,\frac{2\zeta(2\lambda)}{(2\pi^2)^\lambda}$, with $\zeta$ being the Riemann zeta function.

In our Helmholtz PDE problem, the integrand is given by
\[
 F(\bsy) = G\big(u_{s,h}(\cdot,\bsy)\big).
\]
To apply the relevant QMC theory we need to obtain a bound on the norm
$\|F\|_{s,\bsgamma} = \|G(u_{s,h})\|_{s,\bsgamma}$. Using linearity and
boundedness of~$G$, we have
\begin{align} \label{eq:lin-G}
 \bigg|\frac{\partial^{|\setu|}}{\partial \bsy_\setu} G(u_{s,h}(\cdot,\bsy))\bigg|
 \,=\, \bigg|G\bigg(\frac{\partial^{|\setu|}}{\partial \bsy_\setu} u_{s,h}(\cdot,\bsy)\bigg)\bigg|
 \,\le\, \|G\|_{V^*}\ \bigg\| \frac{\partial^{|\setu|}}{\partial \bsy_\setu} u_{s,h}(\cdot,\bsy)\bigg\|_V.
\end{align}
Now applying Theorem~\ref{theorem:bound} with $u$ replaced with $u_{s,h}$
and restricting to multi-indices $\bsnu$ with $\nu_j \leq 1$, we obtain
\begin{align} \label{eq:Gu-norm}
  \|G(u_{s,h})\|_{s,\bsgamma}
 \,\le\, \frac{C_{\rm func}(kL)}{C_{\rm coer}} \|\sbinom{f}{g}\|_{L^2(D) \times L^2(\partial D)}\,\|G\|_{V^*}\,
 \Bigg(
  \sum_{\setu\subseteq\{1:s\}} \frac{(|\setu|!)^2
  \prod_{j\in\setu} \Upsilon_j^2}{\gamma_\setu}
  \Bigg)^{1/2}\,.
\end{align}
The bound \eqref{eq:Gu-norm} takes exactly the same form as in the
diffusion case in \cite{KSS12}, so we could follow the same line of
argument there. Here instead we use a slightly simpler and shorter
argument.

Substituting \eqref{eq:Gu-norm} into the bound \eqref{eq:qmc1} and then
choosing the weights $\gamma_\setu$ to equate the expressions inside the
two sums, we obtain
\begin{equation} \label{eq:weight1}
  \gamma_\setu \,=\,
  \bigg(|\setu|!\, \prod_{j\in\setu} \frac{\Upsilon_j}{\sqrt{\rho(\lambda)}} \bigg)^{2/(1+\lambda)}\,,
  \qquad \Upsilon_j \,=\, C_{\rm regu}(kL)\, \|\psi_j\|_{W^{1,\infty}(D)},
\end{equation}
and this yields
\begin{align*} 
  \sqrt{\bbE_{\rm rqmc} \big[ \big|I_s(G(u^s_h)) -  Q_{s,N}(G(u^s_h);\cdot)\big|^2 \big]}
  \,\le\,
  \frac{C_{s,\bsgamma}(\lambda)}{N^{1/(2\lambda)}}\,
  \frac{C_{\rm func}(kL)}{C_{\rm coer}} \|\sbinom{f}{g}\|_{L^2(D) \times L^2(\partial D)}\,\|G\|_{V^*}\,,
\end{align*}
with
\begin{align*} 
  C_{s,\bsgamma}(\lambda) \,:=\,
  2^{\frac{1}{2\lambda}}
  \Bigg(\sum_{\setu\subseteq\{1:s\}} \bigg(|\setu|!\,\prod_{j\in\setu} \Big(\Upsilon_j\,[\varrho(\lambda)]^{1/(2\lambda)}\Big)
  \bigg)^{\frac{2\lambda}{1+\lambda}}
  \Bigg)^{\frac{1+\lambda}{2\lambda}}.
\end{align*}

We proceed to choose the parameter $\lambda$ such that
$C_{s,\bsgamma}(\lambda)$ is bounded independently of $s$. Since $C_{\rm
regu}(kL) = \calO\big(kL +(kL)^{-1}\big)$, from \eqref{eq:K1} in
Assumption (A5) we know that
\[
  \sum_{j\ge 1} \Upsilon_j^{p_1}
  \,\le\, r_1\,K_1 \,<\, \infty\,,
\]
for some constant $r_1$ independent of the wavenumber $k$. Writing
$\theta_j := \Upsilon_j\,[\varrho(\lambda)]^{1/(2\lambda)}$ and $\tau :=
\frac{2\lambda}{1+\lambda}$, we have
\begin{align*} 
  \sum_{\setu\subseteq\{1:s\}} \bigg(|\setu|!\,\prod_{j\in\setu} \theta_j \bigg)^\tau
  \,=\, \sum_{\ell=0}^s (\ell!)^\tau \sum_{\satop{\setu\subseteq\{1:s\}}{|\setu|=\ell}} \prod_{j\in\setu} \theta_j^\tau
  \,\le\, \sum_{\ell=0}^s (\ell!)^{\tau-1}
  \bigg(\sum_{j=1}^s \theta_j^\tau\bigg)^\ell,
\end{align*}
where the inequality holds because each term $\prod_{j\in\setu}
\theta_j^\tau$ from the left-hand side of the inequality appears in the
expansion $(\sum_{j=1}^s \theta_j^\tau)^\ell$ exactly $\ell!$ times, and
the expansion contains other terms. By the ratio test, the right-hand side
is bounded independently of $s$ provided that $\sum_{j=1}^\infty
\theta_j^\tau <\infty$ and $\tau<1$. Thus in our case we require
$p_1\le\tau<1$, i.e.,
\[
  p_1 \le \frac{2\lambda}{1+\lambda} < 1
  \quad\iff\quad \frac{p_1}{2-p_1} \le \lambda < 1.
\]
Noting that $\lambda$ also needs to satisfy $\frac{1}{2} <\lambda \le 1$,
we therefore choose
\begin{equation*} 
 \lambda \,=\,
 \begin{cases}
 \displaystyle\frac{1}{2-2\delta} \quad\mbox{for some } \delta\in (0,\tfrac{1}{2})
 & \mbox{when } p_1\in (0,\tfrac{2}{3}]\,, \vspace{0.1cm} \\
 \displaystyle\frac{p_1}{2-p_1} & \mbox{when } p_1\in (\tfrac{2}{3},1)\,,
 \end{cases}
\end{equation*}
This leads to the convergence rate $\calO\big( N^{-\min(\frac{1}{p_1}-
\frac{1}{2} , 1-\delta)}\big)$, with the implied constant independent of
$s$.

Weights of the form \eqref{eq:weight1} are known as POD weights (``product
and order dependent weights''). The CBC construction of lattice generating
vector can be done for POD weights in $\calO(s\, N\log N + s^2N)$
operations, see \cite{KSS12}.

Combining the estimates from this subsection with \eqref{eq:trunc-unif}
and \eqref{eq:duality-unif}, we obtain the first main conclusion of this
paper.

\begin{theorem} \label{thm:main-rqmc}
Let the assumptions \textnormal{(A0)--(A5)} and parameter restrictions in
Theorem~\ref{COERTHM} hold. For each $\bsy \in U$, let $u(\cdot, \bsy) \in
V$ be the unique solution of~\eqref{eq:ses_var} and $u_{s,h}(\cdot, \bsy)
\in V_h^p$ be the unique solution of~\eqref{eq:trun_fem_ses_var}. Then for
every $f\in L^2(D)$ and $g\in L^2(\partial D)$, and every linear
functional $G\in V^*$, a generating vector can be constructed for a
randomly shifted lattice rule such that
\begin{align*}
 &\sqrt{\bbE_{\rm rqmc} \big[ \big|I(G(u)) -  Q_{s,N}(G(u_{s,h});\cdot)\big|^2 \big] } \\
  &\,\le\, C\cdot\big(1+(kL)^{-1}\big)\, \bigg(
  s^{- \frac{2}{p_0} + 1} + \big(kL+1\big)\,h^{p} + N^{-\min(\frac{1}{p_1}-\frac{1}{2},1-\delta)}\bigg)\,,
  \quad\delta\in (0,\tfrac{1}{2})\,,
\end{align*}
where $C$ depends on $f$, $g$, $G$, but is independent of $s$, $h$, $N$
and the wavenumber $k$.
\end{theorem}

\subsection{Interlaced polynomial lattice rules (higher order convergence)}

In this subsection we briefly outline the results when we replace randomly
shifted lattice rules by deterministic \emph{interlaced polynomial lattice
rules}, which allow us to obtain higher order convergence rates. The
description below follows closely \cite{DKLNS14}.

Without giving the full technical details, we simply say here that
\eqref{eq:qmc-rsh} is now replaced by a deterministic quadrature rule
\begin{align*}
  Q_{s,N}(F)
  \,=\, \frac{1}{N} \sum_{i=1}^N F\big(\bst_i - \bshalf\big)\,,
\end{align*}
where the points $\bst_i \in [0,1]^s$ are obtained by ``interlacing'' the
points of a ``polynomial lattice rule'', which are specified by a
generating vector of ``polynomials'' rather than of integers. For the
precise details as well as implementation, see e.g., \cite{DKLNS14,KN16}
and the references there. The error bound \eqref{eq:qmc1} is now replaced
by
\begin{align*} 
  \big|I_s(F) - Q_{s,N}(F)\big|
  \,\le\, \bigg(\frac{2}{N} \sum_{\emptyset\ne\setu\subseteq\{1:s\}} \gamma_\setu^\lambda\,
  [\varrho_\alpha(\lambda)]^{|\setu|} \bigg)^{1/(2\lambda)}\,\|F\|_{s,\alpha,\bsgamma}
  \quad\forall\;\lambda\in (\tfrac{1}{\alpha},1],
\end{align*}
where $\alpha\ge 2$ is an integer smoothness parameter (also known as the
``interlacing factor''), $N$ is a power of~$2$, $\varrho_\alpha(\lambda)
  = 2^{\alpha\lambda(\alpha-1)/2} [(1 + \frac{1}{2^{\alpha\lambda}-2})^\alpha]$.
and the norm is now
\begin{align*} 
  \|F\|_{s,\alpha,\bsgamma}
 :=
 \sup_{\setu\subseteq\{1:s\}}
 \sup_{\bsy_\setv \in [0,1]^{|\setv|}}
 \frac{1}{\gamma_\setu}
 \sum_{\setv\subseteq\setu} \,
 \sum_{\bstau_{\setu\setminus\setv} \in \{1:\alpha\}^{|\setu\setminus\setv|}}
 \bigg|\int_{[-\frac{1}{2},\frac{1}{2}]^{s-|\setv|}}
 (\partial^{(\bsalpha_\setv,\bstau_{\setu\setminus\setv},\bszero)} F)(\bsy) \,
 \rd \bsy_{\{1:s\} \setminus\setv}
 \bigg|\,.
\end{align*}

Using again \eqref{eq:lin-G} and Theorem~\ref{theorem:bound} (this time
with general multi-indices), we obtain instead of \eqref{eq:Gu-norm},
\begin{align*}
 \|G(u_{s,h})\|_{s,\alpha,\bsgamma}
 & \,\le\,
 \frac{C_{\rm func}(kL)}{C_{\rm coer}} \|\sbinom{f}{g}\|_{L^2(D) \times L^2(\partial D)}\,\|G\|_{V^*}
 \sup_{\setu\subseteq\{1:s\}}
 \frac{1}{\gamma_\setu}
 \sum_{\bsnu_\setu \in \{1:\alpha\}^{|\setu|}}
 |\bsnu_\setu|!\,\prod_{j\in\setu} \big(2^{\delta(\nu_j,\alpha)} \Upsilon_j^{\nu_j}\big)\,,
\end{align*}
where $\delta(\nu_j,\alpha)$ is $1$ if $\nu_j=\alpha$ and is $0$
otherwise. We now choose $\gamma_\setu$ so that the supremum is $1$, i.e.,
\begin{equation} \label{eq:spod}
 \gamma_\setu
 \,=\,
 \sum_{\bsnu_\setu \in \{1:\alpha\}^{|\setu|}}
 |\bsnu_\setu|!\,\prod_{j\in\setu} \big(2^{\delta(\nu_j,\alpha)} \Upsilon_j^{\nu_j}\big)\,.
\end{equation}
Using the above weights and following the arguments in
\cite[Pages~2694--2695]{DKLNS14}, by taking $\lambda = p_1$ and the
interlacing factor $\alpha = \lfloor 1/p_1\rfloor + 1$, we eventually
arrive at the convergence rate $\calO(N^{-1/p_1})$, with the implied
constant independent of $s, h, N$.

Weights of the form \eqref{eq:spod} are called SPOD weights
(``smoothness-driven product and order dependent weights''). The
generating vector (of polynomials) can be obtained by a CBC construction
in $\calO(\alpha\,s\, N\log N + \alpha^2\,s^2 N)$ operations, see
\cite{DKLNS14}.

We summarize our second main conclusion in the following theorem.

\begin{theorem} \label{thm:main-hqmc}
Let the assumptions \textnormal{(A0)--(A5)} and parameter restrictions in
Theorem~\ref{COERTHM} hold. For each $\bsy \in U$, let $u(\cdot, \bsy) \in
V$ be the unique solution of~\eqref{eq:ses_var} and $u_{s,h}(\cdot, \bsy)
\in V_h^p$ be the unique solution of~\eqref{eq:trun_fem_ses_var}. Then for
every $f\in L^2(D)$ and $g\in L^2(\partial D)$, and every linear
functional $G\in V^*$, a generating vector can be constructed for an
interlaced polynomial lattice rule with interlacing factor $\alpha =
\lfloor 1/p_1\rfloor + 1 \geq 2$ such that
\begin{align*}
 \big|I(G(u)) -  Q_{s,N}(G(u_{s,h}))\big|
  \,\le\, C\cdot\big(1+(kL)^{-1}\big)\, \bigg(
  s^{- \frac{2}{p_0} + 1} + \big(kL+1\big)\,h^{p} + N^{-\frac{1}{p_1}}\bigg)\,,
\end{align*}
where $C$ depends on $f$, $g$, $G$, but is independent of $s$, $h$, $N$
and the wavenumber $k$.
\end{theorem}

\section*{Acknowledgements}

We sincerely thank the anonymous referees for insightful comments and
suggestions which helped to improve the paper. We gratefully acknowledge
the financial support from the Australian Research Council for the project
DP180101356.

\appendix

\section{Small perturbation approach}\label{append:pert}

In this section, by a partial differential operator (PDO) associated with
a boundary value problem (BVP), governed by a PDE and a boundary condition
(BC), we mean the PDO in its weak sense. The weak PDO (WPDO) is a linear
operator induced by a sesquilinear form associated with an equivalent weak
formulation  of the BVP (WBVP).

The stochastic wave propagation Helmholtz PDE model introduced in
Section~\ref{sec:intro} can be reformulated, using the  celebrated
standard weak form, as
\begin{equation}\label{eq:var_overview}
  \bilins_{\bsy}(u,w) \,=\, \antilins(w)
  \qquad \mbox{for all}\quad \bsy\in U\,,\quad w\in H^1(D),
\end{equation}
where, for  fixed $\bsy$,  $\bilins_{\bsy}: H^1(D) \times H^1(D)
\rightarrow \mathbb{C}$ is a sesquilinear form, and $\antilins$ is a
linear functional. More precisely,
\begin{equation*} 
\bilins_{\bsy}(v,w)  =  a_0(v,w) +  \sum_{j\geq 1} y_j\, a_j(v,w),
\end{equation*}
where  for  $v, w \in H^1(D)$
\begin{equation*} 
a_0(v,w) = \int_{D} \Big[ \nabla v \cdot \overline{\nabla w} - k^2\,n_0\,v\,\overline{w}\Big]
-\ii k \int_{\partial D} \gamma v  \overline{\gamma w}, \qquad
a_j(v,w) =  k^2  \int_{D}\, \psi_j v \, \overline{w}, \qquad j \geq 1.
 \end{equation*}
It is well known that $a_0$ is sign-indefinite (that is, non-coercive).
However, $a_0$ satisfies the inf-sup condition with  inf-sup constant
$\mu_0 = \mathcal{O}(1/k)$, see for example~\cite[Cor. 1.10]{BSW16}.
Indeed, till recently, all known and analyzed variational reformulations
of the heterogeneous media deterministic Helmholtz model are
sign-indefinite, see for example \cite{GanMor20} and references therein.

The inf-sup property of $a_0$ has been established~\cite{BSW16}, using the
following weighted ($k$-dependent) norm in $H^1(D)$~\cite{BSW16}:
 \begin{equation*} 
 \| v \|_{H^1_k}^2 = \|\nabla v \|_{L^2(D)}^2  + k^2\,\| v \|_{L^2(D)}^2.
 \end{equation*}
The framework in~\cite{DKLNS14,DKLS16} is established for a general class
of operators defined on reflexive Banach spaces $X, Y$. For our wave
propagation model, it is appropriate to consider linear operators
 $\Lajs: X \rightarrow Y'$ defined as
 \begin{equation*} 
 {}_Y \hspace{-0.05in}\left<w, \Lajs v \right >_{Y'} = a_j(v,w), \qquad  v \in X, \quad w \in Y,  \qquad j \geq 0,
\end{equation*}
with $X = Y = H^1_k(D)$. Consequently the standard
WBVP~\eqref{eq:var_overview} based WPDO $\mathcal{A}$ of the
BVP~\eqref{eq:pde_omega} with PDO $\La$ in~\eqref{eq:pdeoper_omega} is:
\begin{equation}\label{eq:pde_ysplit}
 \mathcal{A}(\bsy) \,=\,
 \Lazs + \sum_{j\geq 1} y_j \Lajs,   \qquad \bsy \in U.
\end{equation}
Thus, thanks to the inf-sup property of $a_0$, we have $\Lazs \in
\mathcal{L}(X, Y')$ is boundedly invertible with $\|\Lazs^{-1}\| =
\mathcal{O}(k)$ and $\|\Lajs\| = \mathcal{O}(1)$, since
\[
  \|\Lajs\| \,\le\, \sup_{v,w\ne 0} \frac{|a_j(v,w)|}{\|v\|_X\,\|w\|_X}
  \,\le\, \frac{\|\psi_j\|_{L^\infty(D)}\,k^2\, \|v\|_{L^2(D)}\,\|w\|_{L^2(D)}}{\|v\|_X\,\|w\|_X}
  \,\le\, \|\psi_j\|_{L^\infty(D)}.
\]

The class of stochastic WPDOs considered in \cite{DKLNS14,DKLS16} are of
the form in~\eqref{eq:pde_ysplit}. The framework in \cite{DKLNS14,DKLS16}
starts with the summability assumption~\cite[Equation (1.3)]{DKLNS14}
\begin{equation*} 
\sum_{j\geq 1} \|\Lajs\|_{\mathcal{L}(X, Y')}^p < \infty, \qquad \text{for some} \quad p \in (0,1],
\end{equation*}
and bounded invertible assumption of $\Lazs$, as a linear operator from
$X$ to $Y'$.

The analysis in \cite{DKLNS14,DKLS16} and related papers, while of wide
generality, requires that the operator sum in \eqref{eq:pde_ysplit} be
small, in the sense that
\begin{equation*} 
 \Lazs + \sum_{j\geq 1}  y_j \Lajs
 \,=\, \Lazs\Big( \mathcal{I} +\sum_{j\geq 1} y_j \Lazs^{-1} \Lajs \Big)
\end{equation*}
should satisfy, using $\left |y_j \right | \leq 1/2$,
\begin{equation}\label{eq:neu-ass}
\frac{1}{2} \sum_{j\geq 1} \|\Lazs^{-1} \Lajs \| \,<\, 1\,,
\end{equation}
since if this is satisfied then the Neumann series for the inverse of the
operator sum converges in operator norm in the space $X$. Accordingly, it
seems reasonable to say that any argument based on \eqref{eq:neu-ass} is
using the ``small perturbation'' approach. Note that \eqref{eq:neu-ass},
when applied to our wave propagation model, requires that a quantity of
the order $k \sum_{j\geq 1} \left\| \psi_j \right\|_{L^\infty(D)}$ be less
than~$1$.

\section{Technical lemma}

\begin{lemma} \label{lem:recur4}
Given some non-negative real numbers $(\Psi_j)_{j\in\bbN}$ and constants
$c_0,c_1,c_2,B$, let $(\bbA_\bsnu)_{\bsnu\in\indx}$ be non-negative real
numbers satisfying the inequality
\begin{equation*} 
  \bbA_\bsnu
  \,\le\,
  \begin{cases}
  B & \mbox{if } \bsnu = \bszero, \\
  c_0\,\Psi_j\, B & \mbox{if } \bsnu = \bse_j, \vspace{0.2cm} \\
  c_1\,\displaystyle\sum_{j\in\supp(\bsnu)} \nu_j\,\Psi_j\, \bbA_{\bsnu-\bse_j} \\
  + c_2\,\displaystyle\sum_{j\in\supp(\bsnu)}\sum_{\ell\in\supp(\bsnu-\bse_j)}
  \nu_j\,(\bsnu-\bse_j)_\ell\,
  \Psi_j\, \Psi_\ell\, \bbA_{\bsnu-\bse_j-\bse_\ell}
  & \mbox{if } |\bsnu| \ge 2.
  \end{cases}
\end{equation*}
Then for any $\bsnu \in \indx$ we have
\begin{equation*} 
  \bbA_\bsnu
  \,\le\, |\bsnu|!\,\bsUpsilon^\bsnu\,B\;,
  \quad\mbox{with}\quad \bsUpsilon^\bsnu := \prod_{j\ge 1} \Upsilon_j^{\nu_j},
  \qquad \Upsilon_j :=
  \max\left\{c_0, 2c_1,\sqrt{2c_2}\right\} \,\Psi_j
  \;.
\end{equation*}
\end{lemma}

\begin{proof}
Let $\Upsilon_j = C\, \Psi_j$. We prove this result by induction while
determining the multiplying factor~$C$. The cases $|\bsnu|\le 1$ hold
trivially if $c_0\le C$. Suppose that the result holds for all $|\bsnu| <
n$ with some $n\ge 1$. Then for $|\bsnu| = n\ge 2$, we can split the terms
in the inequality into
\begin{equation*}
  \bbA_\bsnu
  \,\le\, c_1\,\sum_{j\ge1} \nu_j\,\Psi_j\, \bbA_{\bsnu-\bse_j}
  + c_2\,\sum_{j\ge1}
  \nu_j\,(\nu_j-1)\, \Psi_j^2\, \bbA_{\bsnu-2\bse_j}
  + c_2\,\sum_{j\ge1}\sum_{\satop{\ell\ge1}{\ell\ne j}}
  \nu_j\,\nu_\ell\, \Psi_j\, \Psi_\ell\, \bbA_{\bsnu-\bse_j-\bse_\ell}.
\end{equation*}
Applying the induction hypothesis then leads to
\begin{align*}
  \bbA_\bsnu
  &\,\le\,
  c_1\,\sum_{j\ge1} \nu_j\,\Psi_j\,
  (|\bsnu|-1)!\, \bsUpsilon^{\bsnu-\bse_j} \,B
  + c_2\,\sum_{j\ge1} \nu_j\,(\nu_j-1)\, \Psi_j^2\,
  (|\bsnu|-2)!\, \bsUpsilon^{\bsnu-2\bse_j} \,B \\
  &\qquad\qquad\qquad\qquad\qquad\qquad\quad
  + c_2\,\sum_{j\ge1} \sum_{\satop{\ell\ge1}{\ell\ne j}}
  \nu_j\,\nu_\ell\, \Psi_j\, \Psi_\ell\,
  (|\bsnu|-2)!\, \bsUpsilon^{\bsnu-\bse_j-\bse_\ell}\, B \\
  &\,\le\,
  \frac{c_1}{C}\,\sum_{j\ge1} \nu_j
  (|\bsnu|-1)!\, \bsUpsilon^{\bsnu} \,B
  + \frac{c_2}{C^2}\,\sum_{j\ge1} \nu_j\,(\nu_j-1)\,
  (|\bsnu|-2)!\, \bsUpsilon^{\bsnu} \,B \\
  &\qquad\qquad\qquad\qquad\qquad\qquad\quad
  + \frac{c_2}{C^2}\,\sum_{j\ge1} \sum_{\satop{\ell\ge1}{\ell\ne j}}
  \nu_j\,\nu_\ell\,
  (|\bsnu|-2)!\, \bsUpsilon^{\bsnu}\, B
  \,=\, \left(\frac{c_1}{C}+\frac{c_2}{C^2}\right)|\bsnu|!\, \bsUpsilon^{\bsnu}\, B.
\end{align*}

If $c_1 \le \frac{C}{2}$ and $c_2 \le \frac{C^2}{2}$, then $\frac{c_1}{C}
+ \frac{c_2}{C^2}\le 1$. So we may choose $C := \max\{c_0,2c_1,
\sqrt{2c_2}\}$ as stated in the lemma.

An alternative bound can be obtained by choosing $C := \max\{c_0,c_3\}$,
with $c_3 := \frac{c_1 + \sqrt{c_1^2+4c_2}}{2}$ which satisfies
$\frac{c_1}{c_3}+\frac{c_2}{c_3^2}=1$.
\end{proof}

\end{document}